\documentclass[10pt,reqno]{amsart}

\newenvironment{dedication}
{
   \itshape             
   \centering  
}

\usepackage{amsfonts}
\usepackage{amsmath}
\usepackage[colorlinks]{hyperref}
\usepackage{amssymb}
\usepackage{tikz}
\usepackage{color}
\usepackage[all]{xy}
\usepackage{enumerate}
\usepackage{bm}
\usepackage{bbm}
\usepackage[normalem]{ulem}
\usepackage{mathrsfs}
\usepackage{tikz-cd}
\usepackage{mathtools}

\usepackage[margin = 1.25in, centering]{geometry}

\allowdisplaybreaks
\numberwithin{equation}{section}
\newtheorem{theorem}{Theorem}[section]
\newtheorem{proposition}[theorem]{Proposition}
\newtheorem{lemma}[theorem]{Lemma}
\newtheorem{corollary}[theorem]{Corollary}

\theoremstyle{definition}
\newtheorem{definition}[theorem]{Definition}

\newtheorem{question}[theorem]{Question}
\theoremstyle{remark}
\newtheorem{remark}[theorem]{Remark}
\newtheorem{example}[theorem]{Example}

\newcommand{\G}{\mathcal{G}}

\newcommand{\A}{\mathcal{A}}
\newcommand{\B}{\mathcal{B}}
\newcommand{\C}{\mathcal{C}}

\title[The structure of subalgebras of a full matrix algebra.]{The structure of subalgebras of full matrix algebras\\ over a field satisfying the identity $[x_1, y_1][x_2, y_2]\cdots [x_q, y_q] = 0$}

\date{\today}

\author[P. Matra\'s]{Pawe\l{} Matra\'s}
\address{P. Matra\'s: Faculty of Mathematics and Information Science, Technical University of Warsaw, 00-661 Warsaw, Poland}
\email{p.matras@mini.pw.edu.pl}

\author[L. van Wyk]{Leon van Wyk}
\address{L. van Wyk: Department of Mathematical Sciences, Stellenbosch University, Private Bag X1, Matieland 7602, Stellenbosch, South Africa}
\email{lvw@sun.ac.za}

\author[M. Ziembowski]{Micha\l{} Ziembowski}
\address{M. Ziembowski: Faculty of Mathematics and Information Science, Technical University of Warsaw, 00-661 Warsaw, Poland}
\email{michal.ziembowski@pw.edu.pl}

\begin{document}

\maketitle

\begin{dedication}
Dedicated to Sorin D\u asc\u alescu for his sixtieth birthday
\end{dedication}

\begin{abstract} A subalgebra of the full  matrix algebra M$_n(K), \ K$ a field, satisfying the identity 
$\left[x_1, y_1\right]\left[x_2, y_2\right]\cdots\left[x_q, y_q\right] = 0$ 
is called a D$_q$ subalgebra of M$_n(K)$. In the paper we deal with the structure, conjugation and isomorphism problems of maximal D$_q$ subalgebras of~M$_n(K)$.


We show that a maximal D$_q$ subalgebra $\A$ of M$_n(K)$ is conjugated with a block triangular subalgebra of M$_n(K)$ 
with maximal commutative diagonal blocks.  
By analysis of conjugations, the sizes of the obtained diagonal blocks are uniquely determined. It reduces the problem of conjugation of maximal D$_q$ subalgebras of M$_n(K)$ to the analogous problem in the class of commutative subalgebras of M$_n(K)$. Further examining conjugations, in case $\A$ is contained in the upper triangular matrix algebra~U$_n(K)$, we prove that $\A$ is already in a block triangular form.

We consider the isomorphism problem in a certain class of maximal D$_q$ subalgebras of M$_n(K)$ which contain all D$_q$ subalgebras of M$_n(K)$ with maximum dimension.
In case $K$ is algebraically closed, we invoke Jacobson's characterization of maximal commutative subalgebras of 
M$_n(K)$ with maximum ($K$-)dimension to show that isomorphic subalgebras in this class are already conjugated.
To illustrate it, we invoke results from \cite{LM1} and find all 
isomorphism (equivalently conjugation) classes 
of~D$_q$~subalgebras of M$_n(K)$ with maximum possible  dimension, in case $K$ is algebraically closed.




\end{abstract}

\tableofcontents

\section{Motivation, background and results on D$_q$ subalgebras of full matrix algebras} \label{Section 1}


Throughout the sequel, all algebras are assumed to be associative unital and over a field $K$.  By a subalgebra of the full $n\times n $ matrix algebra M$_n(K)$ we mean a $K$-subalgebra of M$_n(K)$. 

The main and direct motivation for the work presented here comes from 
\cite{Domokos}, \cite{LM1} and \cite{LM2}. 
In \cite{Domokos}, Domokos deals with the identity 
\begin{equation}\label{Domokosq}
\left[x_1, y_1\right]\left[x_2, y_2\right]\cdots\left[x_q, y_q\right] = 0
\end{equation}
in the context of subalgebras of M$_n(K)$.
 Here, $[x,y]$ denotes the commutator Lie product $xy-yx$ (also called
the Lie bracket in the literature), and $q$ is a positive integer. 

We note that $\mathrm{M}_{n}(K)$ with the commutator Lie
product plays an exceptional role in the theory of finite-dimensional Lie
algebras. The fundamental Ado-Iwasawa theorem (see \cite{Ja}) asserts that
every finite-dimensional Lie $K$-algebra can be embedded into $\mathrm{M}%
_{n}(K)$ for some $n\geq1$.

Finite dimensional basic algebras
over algebraically closed fields play an important role in the representation theory of Artinian algebras (see \cite{Aus}). Such algebras satisfy (\ref{Domokosq}) for some $q$. An Artinian ring $R$ satisfies (\ref{Domokosq}) for some $q$ if and only if $R/{\rm rad}(R)$ is commutative, in which  case the index of nilpotency of rad$(R)$ is an upper
bound for the least such $q$.

The identity in (\ref{Domokosq}) has featured prominently in many other papers. See, for example,  \cite{BeiMi}, 
\cite{Ma1}, \cite{Ma2} and~\cite{MeSzvW}. It was proved in \cite{Ma1} that all the polynomial identities of the upper triangular $q\times q$ matrix algebra U$_q(K)$ over any field $K$ are consequences of only one identity, namely the identity in~(\ref{Domokosq}). For an explicit form of a finite set of generators of an ideal of identities of 
the algebra~U$_q^{\star}(R)$ 
over a commutative integral domain $R$, see \cite{Ma2}. 
Here U$_q^{\star}(R)$ denotes the subalgebra of U$_q(R)$ comprising all the matrices (in U$_q(R)$) with constant main diagonal. 

The $9\times 9$ matrix algebra ${\rm U}_3^{\star}\bigl({\rm U}_3^{\star}(R)\bigr)$ over any commutative ring $R$ was exhibited in \cite{MeSzvW} as an example of an algebra satisfying the polynomial identity
$[[x_1,y_1],[x_2,y_2]]=0$ (Lie solvability index~two), but none of the stronger identities $[x_1,y_1][x_2,y_2]=0$ (the identity in~(\ref{Domokosq}), with $q=2$) and
$[[x,y],z]=0$ (Lie nilpotency index two).  A Cayley-Hamilton trace identity was exhibited in~\cite{MeSzvW} for $2\times2$
matrices with entries in a ring $R$ satisfying $[x_1,y_1][x_2,y_2]=0$ and $\frac{1}{2}\in R$. See also \cite{JSL}.

The Cayley-Hamilton theorem and the corresponding trace identity play a crucial role (see \cite{Dren} and \cite{DrenForma}) in proving classical results about the polynomial and trace identities of $\mathrm{M}_{n}(K)$. In case $\mathrm{char}(K)=0$, Kemer's pioneering
work (see \cite{Kemer}) on the T-ideals of associative algebras revealed the significance of the identities satisfied by the $n\times n$ matrices over the Grassmann (exterior) algebra
%
generated by an infinite sequence of anticommutative indeterminates.

If an algebra satisfies~$(\ref{Domokosq})$, then we say that it is D$_q$, and if a subalgebra of M$_n(K)$ is D$_q$, then we say that it is a D$_q$~subalgebra of M$_n(K)$. 

Considering D$_1$, i.e., when $q=1$, we get exactly commutativity, which in the 
context of subalgebras of M$_n(K)$, features prominently in the cited literature (see, for example, \cite{Jac} and \cite{Sch}). In particular, a classical result by Schur (see \cite{Sch}) states that the maximum $K$-\,dimension of a commutative subalgebra of M$_n(K)$, with $K$ an algebraically closed field, is $\left \lfloor \frac{n^2}{4} \right \rfloor+1$.
Here $\lfloor \; \rfloor$ denotes the integer floor function.
Schur's result was extended to an arbitrary field by Jacobson in \cite{Jac}. We  often  write dimension instead of $K$-\,dimension.

  If we say that an algebra $\A$ is maximal in an algebra $\mathcal{E}$ with respect to some conditions then we think about the inclusion relation. If, in the context of some class of subalgebras of a finite dimensional  algebra $\mathcal{E}$, we say that  $\A$ has maximum dimension, then we mean that $\A$ has maximum possible dimension in the considered class. 

The mentioned maximum dimension $\left \lfloor \frac{n^2}{4} \right \rfloor+1$ of a commutative subalgebra of M$_n(K)$ is obtained by considering the subalgebra
\begin{equation} \label{typical commutative 1}
K I_n + 
\left(
\begin{array}{cc}
0_\ell & {\rm M}_\ell(K) \\
0_\ell & 0_\ell
\end{array}
\right)
\end{equation}
of M$_n(K)$ if $n$ is even (with $n=2\ell$, for some integer $\ell$), and by considering the subalgebra 
\begin{equation} \label{typical commutative 2}
K I_n + 
\left(
\begin{array}{cc}
0_\ell                & {\rm M}_{\ell \times (\ell+1)}(K) \\
0_{(\ell+1) \times \ell} & 0_{\ell+1}
\end{array}
\right) 
\end{equation}
of M$_n(K)$ if $n$ is odd (with $n=2\ell+1$).
Here, for example, $0_\ell$ and  $0_{(\ell+1) \times \ell}$ denote the $\ell \times \ell$ and $(\ell+1) \times \ell$ zero matrices, respectively.

\medskip

Henceforth, when we consider a D$_q$ subalgebra $\A$ of M$_n(K)$, then we always assume that $q\ge 2$ and that $\A$ is not D$_{q-1}$ (and hence not D$_1$). 
 
 \medskip

The main result in \cite{Domokos} is the following:
\begin{theorem} (\ \negthinspace\cite[Theorem 1]{Domokos}) \label{Domokos}
Let $K$ be a field, and $\mathcal{A}$ a finite dimensional  $K$-algebra satisfying~$(\ref{Domokosq})$.
If $M$ is a finitely generated faithful module over $\mathcal{A}$,
then 
\begin{equation}\label{Domokoseq1}
{\rm dim}_KM\geq \sqrt{\frac{{\rm dim}_K\mathcal{A}-q}{\frac{1}{2} - \frac{1}{4q}}}.
\end{equation}
\end{theorem}



\bigskip

In the proof of  the above theorem,  
Domokos shows that, for the considered $K$-algebra~$\mathcal{A}$, 
\begin{equation}\label{Domokoseq2}
{\rm dim}_K\mathcal{A} \leq \frac{1}{2}({\rm dim}_KM)^2 + q - \sum_{i=1}^q\left(\frac{n_i^2}{2} - \left\lfloor \frac{n_i^2}{4}\right\rfloor \right) 
\end{equation}
for some  positive integers $n_1, \ldots, n_q$ such that $n_1+\cdots+n_q = {\rm dim}_KM$. 

\medskip

If one takes $M=K^n$, then the right hand side in $(\ref{Domokoseq2})$ takes the form
\begin{equation}\label{Domokoseq4}
\frac{1}{2}\left(n_1 + \cdots + n_q\right)^2 + q - \sum_{i=1}^q\left(\frac{n_i^2}{2} - \left\lfloor \frac{n_i^2}{4}\right\rfloor \right). 
\end{equation}
which equals the  expression in (\ref{Domokoseq3}) below, implying that the inequality in $(\ref{Domokoseq2})$ is sharp. 

In \cite{Domokos}, the $n_i$'s are mentioned as any numbers  which guarantee that 
$$\frac{1}{2}({\rm dim}_KM)^2 + q - \sum_{i=1}^q\left(\frac{n_i^2}{2} - \left\lfloor \frac{n_i^2}{4}\right\rfloor \right)$$ is a maximum. Evidently, such an $q$-tuple $(n_1, n_2,\ldots, n_q)$ exists, but it is not exhibited in~\cite{Domokos}. 
In this regard, we refer the reader to 
\cite{LM1}, where such an $q$-tuple is explicitly described and the maximum is exhibited precisely:

\begin{theorem}\label{precise} (\cite[Theorem 14]{LM1})
   Let $1\leq q\leq n$, and let $n = q\left\lfloor \frac{n}{q} \right\rfloor + r$, $0\leq r< q$ 
   (with~$r$ as in the Division Algorithm). Then the precise sharp upper bound for the dimension of a {\rm D}$_q$ subalgebra of ${\rm M}_n(K)$ is
 \begin{equation}\label{maxdimDq}  
\frac{1}{2}\left(n^2 - \left(q-r\right)\left\lfloor \frac{n}{q}\right\rfloor^2 - r\left( \left\lfloor \frac{n}{q}\right\rfloor + 1\right)^2\right) +q + \left(q-r\right)\left\lfloor\frac{\left\lfloor \frac{n}{q}\right\rfloor^2}{4}\right\rfloor + r\left\lfloor\frac{\left( \left\lfloor \frac{n}{q}\right\rfloor + 1\right)^2}{4}\right\rfloor,
\end{equation}
which can be obtained by choosing $q-r$ commutative subalgebras of ${\rm M}_{\left\lfloor \frac{n}{q}\right\rfloor}(K)$ of dimension $\left\lfloor\frac{\left\lfloor \frac{n}{q}\right\rfloor^2}{4}\right\rfloor + 1$ and $r$ commutative subalgebras of ${\rm M}_{\left\lfloor \frac{n}{q}\right\rfloor+1}(K)$ of dimension
$\left\lfloor\frac{\left(\left\lfloor \frac{n}{q}\right\rfloor+ 1\right)^2}{4}\right\rfloor + 1$ on the diagonal blocks for the algebra presented in (\ref{3type}) 
(see also \cite[page 157]{Domokos}).
   \end{theorem}

\medskip

In this vein we also draw the reader's attention to \cite{JJLM}, where the maximum dimension of a Lie nilpotent subalgebra of M$_n(K)$ of index $m$ is obtained.

In general, if $\mathcal{A}$ is a subalgebra of M$_n(K)$ and  every matrix $A \in \mathcal{A}$ is seen in the block triangular form
$$A=
\left(
\begin{array}{ccccc}
A_{11} & A_{12}  & \ldots      & A_{1q} \\ 
       & A_{22}  & \ldots      & A_{2q}  \\ 
       &         & \ddots      & \vdots   \\ 
       &         &             &   A_{qq}            
\end{array}
\right),$$
where $A_{ij} \in {\rm M}_{n_i \times n_j}(K)$ for all $i\leq j$, then considering 
the set
\begin{equation} \label{block}
\overline{\mathcal{A}}_{ii} = \left \{ A_{ii} \in {\rm M}_{n_i}(K) :      
\left(
\begin{array}{ccccc}
A_{11} & \ldots  & A_{1i} & \ldots   & A_{1q} \\ 
  &      \ddots  & \vdots  & \ddots  & \vdots  \\ 
  &              & A_{ii}  & \ldots  & A_{iq}   \\ 
  &              &         & \ddots  & \vdots \\ 
  &              &         &         & A_{qq}            
\end{array}
\right) \in \mathcal{A} 
\right \},
\end{equation} it is important to note that 
\begin{equation} \label{blocks2}
\left(
\begin{array}{ccccccc}
0_{n_1}   &\cdots  & 0_{n_1 \times n_{i-1}} & 0_{n_1 \times n_i}      &  0_{n_1 \times n_{i+1}}     & \cdots & 0_{n_1 \times n_q} \\
          & \ddots & \vdots                 & \vdots                  & \vdots                      & \ddots & \vdots            \\
          &        & 0_{n_{i-1}}            & 0_{n_{i-1} \times n_i}  &  0_{n_{i-1} \times n_{i+1}} & \cdots & 0_{n_{i-1} \times n_q} \\
          \\
          &        &                        & \overline{\mathcal{A}}_{ii}       & 0_{n_i \times n_{i+1}}      & \cdots  & 0_{n_i \times n_q} \\ 
          \\
          &        &                        &                         & 0_{n_{i+1}}                 & \cdots  & 0_{n_{i+1} \times n_q} \\ 
          &        &                        &                         &                             & \ddots  & \vdots   \\ 
          &        &                        &                         &                             &         & 0_{n_q}            
\end{array}
\right)
\end{equation}
need not be a subset of $\mathcal{A}$.
If the set in (\ref{blocks2}) is indeed contained in $\A$ for every $i, \ i=1,2,\ldots,q$, then we say that the algebras $\overline{\mathcal{A}}_{ii}$ are {\it independent}.

For example, every matrix $A$ in the subalgebra
\[
\left\{\left(
\begin{array}{ccccccccc}
a & b & c & e & f & g & t & u & v \\
0 & a & d & 0 & e & h & 0 & t & w \\
0 & 0 & a & 0 & 0 & e & 0 & 0 & t \\
0 & 0 & 0 & a & b & c & p & q & r \\
0 & 0 & 0 & 0 & a & d & 0 & p & s \\
0 & 0 & 0 & 0 & 0 & a & 0 & 0 & p \\
0 & 0 & 0 & 0 & 0 & 0 & a & b & c \\
0 & 0 & 0 & 0 & 0 & 0 & 0 & a & d \\
0 & 0 & 0 & 0 & 0 & 0 & 0 & 0 & a \\
\end{array}
\right): a,b,c,d,e,f,g,h,p,q,r,s,t,u,v,w\in K\right\}
\]
of U$_9(K)$ in \cite[Corollary 2.2]{MeSzvW} can be written in the block triangular form
\[
A=
\left(
\begin{array}{ccc}
A_{11}       & A_{12}  &  A_{13} \\ 
             & A_{22}  & A_{23}  \\ 
             &         & A_{33}            
\end{array}
\right),
\]
with $A_{ij}\in {\rm M}_3(K)$ for $1 \leq i \leq j \leq q$. In this example the three algebras $\overline{\A}_{11}$, $\overline{\A}_{22}$ and $\overline{\A}_{33}$ are not independent.
\medskip

\medskip
We will consider block triangular subalgebras of M$_n(K)$ where the sizes of the diagonal blocks play an important role:

\begin{definition}\label{of type}
For any positive integers $n_1, \ldots, n_q$, with $q\ge2$, such that $n_1 + \cdots + n_q = n$, consider a  block triangular subalgebra
\begin{equation} \label{block triang}
\mathcal{A}=
\left(
\begin{array}{ccccc}
\mathcal{A}_{11}   & \mathcal{A}_{12}    & \ldots      & \mathcal{A}_{1q} \\ 
                   & \mathcal{A}_{22}    & \ldots      & \mathcal{A}_{2q}  \\ 
                   &                     & \ddots      & \vdots   \\
                   &                     &             &  \mathcal{A}_{qq}            
\end{array}
\right)
\end{equation}
of M$_n(K)$ 
where 
\begin{enumerate}
\item $\mathcal{A}_{ii}$  is a subalgebra of M$_{n_i}(K)$ for every $i, \ i =1,2, \ldots, q$,  
\item $\mathcal{A}_{ij} = {\rm M}_{n_i \times n_j}(K)$ for all $i$ and $j$ such that $1 \leq i < j \leq q$, and
\item all other entries are zero.
\end{enumerate}


We call $\mathcal{A}$ a 
 subalgebra of~{\rm M}$_n(K)$ of {\it type $(n_1, n_2, \ldots, n_q)$}. If $\A_{ii}={\rm M}_{n_i}(K)$ for all $i$, then we call~$\A$ the {\it full subalgebra of M$_n(K)$ of {\it type $(n_1, n_2, \ldots, n_q)$}}.

\end{definition}

\medskip

It is important to note that the notation in (\ref{block triang}) means that 
for every $i$, $ i = 1, 2, \ldots, q$, 


\begin{equation} \label{diagonal block}
\left(
\begin{array}{ccccccc}
0_{n_1}   &\cdots  & 0_{n_1 \times n_{i-1}} & 0_{n_1 \times n_i}      &  0_{n_1 \times n_{i+1}}     & \cdots & 0_{n_1 \times n_q} \\
          & \ddots & \vdots                 & \vdots                  & \vdots                      & \ddots & \vdots            \\
          &        & 0_{n_{i-1}}            & 0_{n_{i-1} \times n_i}  &  0_{n_{i-1} \times n_{i+1}} & \cdots & 0_{n_{i-1} \times n_q} \\
          \\
          &        &                        & \mathcal{A}_{ii}        & 0_{n_i \times n_{i+1}}      & \cdots  & 0_{n_i \times n_q} \\ 
          \\
          &        &                        &                         & 0_{n_{i+1}}                 & \cdots  & 0_{n_{i+1} \times n_q} \\ 
          &        &                        &                         &                             & \ddots  & \vdots   \\ 
          &        &                        &                         &                             &         & 0_{n_q}            
\end{array}
\right)\subseteq \mathcal{A},
\end{equation} 

\bigskip

\noindent 
and similarly, for all $i$ and $j$ such that $1 \leq i < j \leq q$, the subset of M$_n(K)$ having $\A_{ij}$ ($={\rm M}_{n_i\times n_j}(K)$) in block $(i,j)$, and zeroes elsewhere, is also contained in $\A$.

\bigskip


\begin{remark} \label{typicaldq}
In the case when all the algebras $\A_{ii}, \ i=1,2,\ldots,q$, on the diagonal blocks of a subalgebra $\A$ of M$_n(K)$ of type $(n_1, n_2, \ldots, n_q)$ are commutative, then for any $X, Y \in \A$, the commutator
$[X,Y]$ is an element of 
$$
\left(
\begin{array}{ccccc}
0_{n_1} & {\rm M}_{n_1\times n_2}(K) & {\rm M}_{n_1\times n_3}(K) & \cdots      & {\rm M}_{n_1\times n_q}(K)\\ 
       \\
        & 0_{n_2}              & {\rm M}_{n_2\times n_3}(K) & \cdots      & {\rm M}_{n_2\times n_q}(K)  \\  \\
        \\
        &                      & \ddots               & \ddots      & \vdots                \\ 
        \\
        &                      &                      & \ddots      & {\rm M}_{n_{q-1}\times n_q}(K)\\  
        \\
        &                      &                      &             & 0_{n_q}           
\end{array}
\right).
$$
It follows that the product of any $q$ such commutators is zero, and so $\A$ is D$_q$. 
\end{remark}

This remark enables us to define three classes of D$_q$ subalgebras $\mathcal{A}$ of M$_n(K)$ of type $(n_1, n_2, \ldots, n_q)$ (for some $q$-tuple $(n_1, n_2, \ldots, n_q)$). We stress that, for each of these classes, all the subalgebras in the diagonal blocks are assumed to be commutative. Keeping this in mind, and using ``max-comm", ``max-dim" and ``db's'' as abbreviations for ``maximal commutative", ``maximum dimensional" and ``diagonal blocks'', respectively, we now state:  

\medskip

\begin{definition} \label{abbreviations}
Let $\A$ be a subalgebra of~{\rm M}$_n(K)$ of type $(n_1, n_2, \ldots, n_q)$ (see Definition \ref{of type}), with every subalgebra~$\A_{ii}$ of~M$_{n_i}(K)$ commutative, $\ i =1,2,\ldots,q$. Then $\A$ is called a 
\begin{enumerate}
\item {\rm D}$_q$ subalgebra of~{\rm M}$_n(K)$ of type $(n_1, n_2, \ldots, n_q)$ with max-comm db's if every $\mathcal{A}_{ii}$  is a maximal commutative subalgebra of M$_{n_i}(K)$, \ $i =1,2,\ldots,q$;
\item  {\rm D}$_q$ subalgebra of~{\rm M}$_n(K)$ of type $(n_1, n_2, \ldots, n_q)$ with max-dim db's if every $\mathcal{A}_{ii}$  is a commutative subalgebra of M$_{n_i}(K)$ with maximum dimension (equal to $\left\lfloor \frac{n_i^2}{4}\right\rfloor + 1$), $\ i =1,2,\ldots,q$;
\item max-dim {\rm D}$_q$ subalgebra of~{\rm M}$_n(K)$ of type $(n_1, n_2, \ldots, n_q)$ with max-dim db's if every $\mathcal{A}_{ii}$ is a
commutative subalgebra of M$_{n_i}(K)$ with maximum dimension (equal to $\left\lfloor \frac{n_i^2}{4}\right\rfloor + 1$), $\ i =1,2,\ldots,q$,
and  $\A$ is a D$_q$ subalgebra of M$_n(K)$ of maximum dimension (as in (\ref{maxdimDq})).
\end{enumerate}
\end{definition}

\medskip

Note that, letting $n_1,\ldots, n_q$ as in Theorem \ref{precise}, we obtain an algebra $\A$ as in Definition \ref{abbreviations}(3) by taking a subalgebra

\begin{equation}\label{3type}
\mathcal{A}=
\left(
\begin{array}{ccccc}
\mathcal{A}_{11}   & \mathcal{A}_{12}    & \ldots      & \mathcal{A}_{1q} \\ 
                   & \mathcal{A}_{22}    & \ldots      & \mathcal{A}_{2q}  \\ 
                   &                     & \ddots      & \vdots   \\
                   &                     &             &  \mathcal{A}_{qq}            
\end{array}
\right)
\end{equation}
of M$_n(K)$ constructed in~\cite{Domokos} with
\begin{equation}\label{Domokoseq3}
{\rm dim}_K\mathcal{A} = q+\sum_{i=1}^q\left\lfloor\frac{n_i^2}{4}\right\rfloor + \sum_{1\leq i<j\leq q}n_in_j.
\end{equation}


We draw the reader's attention to the fact that a max-dim D$_2$ subalgebra $\A$ of M$_n(K)$ of some type $(n_1,n_2)$ with max-dim db's such that $\A\subseteq {\rm U}_n(K)$ is called a typical D$_2$ subalgebra of U$_n(K)$ in \cite{LM2}.

\medskip

After preparatory results in Section \ref{section 2}, we prove in Section \ref{section3} and Section \ref{section 4} that, up to conjugation, a subalgebra~$\A$ of M$_n(K)$ is a maximal D$_q$ subalgebra of M$_n(K)$ if and only if $\A$ is a {\rm D}$_q$ subalgebra of~{\rm M}$_n(K)$ of (some) type $(n_1, n_2, \ldots, n_q)$ with max-comm db's (see Theorem \ref{corollary block} and Theorem~\ref{typical are maximal}). 
Continuing our analysis of conjugations, 
we show in Corollary \ref{upper triangular max} that in case $\mathcal{A}$ is a maximal D$_q$ subalgebra of M$_n(K)$ contained in~U$_n(K)$, then  $\mathcal{A}$ is a D$_q$ subalgebra of M$_n(K)$ of some type~$(n_1,n_2,\ldots, n_q)$ with max-comm db's.
By examining in Theorem \ref{max up to conj} when two D$_q$ subalgebras 
of~M$_n(K)$ with max-comm db's are conjugated, we prove that the uniqueness of the mentioned tuple $(n_1, n_2, \ldots, n_q)$ and the pairwise uniqueness (up to conjugation) of the algebras in the corresponding $q$ diagonal blocks are necessary and sufficient conditions. 

Next, we will deal with the isomorphism problem of D$_q$ subalgebras of M$_n(K)$ with max-dim~db's. In Section \ref{isomorphism problem}, after giving an interpretation of these algebras in the light of results from Section~\ref{section3} and~Section~\ref{section 4},  we describe necessary conditions for two D$_q$ subalgebras 
of M$_n(K)$ with max-dim db's to be isomorphic (see Theorem~\ref{theorem typical dk algebras}). Using results from Section \ref{remarks on jacobson}, where we clarify the structure of commutative subalgebras of matrix algebras over  algebraically closed fields, as discussed in \cite{Jac}, we also provide in Theorem \ref{characterisation of typical Dq} sufficient conditions for two D$_q$ subalgebras of M$_n(K)$ with max-dim~db's to be isomorphic, in case the field $K$ is algebraically closed. It turns out that such isomorphic subalgebras are already conjugated. In Section \ref{sectionmaximal}, we illustrate theorems obtained in the previous one section on D$_q$ subalgebras of~M$_n(K)$ with maximum dimension. In order to do it, we recall results in \cite{LM1} about the non-uniqueness of $q$-tuples $(n_1, n_2, n_2, \ldots, n_q)$ which give rise to a max-dim D$_q$ subalgebra of M$_n(K)$.

\section{Block form of subalgebras of M$_n(K)$ with nilpotent ideal} \label{section 2}



In this section we will show (in Lemma \ref{lemma block triangulation}) a  block triangular form of  subalgebras of the matrix algebra M$_n(K)$ containing a nonzero nilpotent ideal (see also \cite[Theorem 1.5.1]{SIM}). We provide a relatively detailed proof of~Lemma~\ref{lemma block triangulation} and illustrate it in Example \ref{example block triangular}.

Lemma~\ref{lemma block triangulation} will be invoked in Section \ref{section3}, where we will prove (in Theorem~\ref{corollary block}) that every maximal~D$_q$ subalgebra of M$_n(K)$ is  conjugated with a {\rm D}$_q$ subalgebra of~{\rm M}$_n(K)$ of type $(n_1, n_2, \ldots, n_q)$ with max-comm db's.

We conclude the section by showing (in Proposition \ref{proposition identities with nilpotent ideal}) that every 
D$_q$ algebra 
contains a nonzero nilpotent ideal in a natural way. 

\begin{lemma} \label{lemma block triangulation}
Let $\mathcal{A}$ be a subalgebra of M$_n(K)$, and let $I$ be a nonzero nilpotent ideal of~$\mathcal{A}$ with nilpotency index $q$, i.e. $I^q = 0$ and $I^{q-1} \neq 0$. 
Then $q \leq n$, and there exist 
natural numbers $n_1, n_2,\ldots, n_q$ such that $\sum_{i=1}^q n_i = n$ and an invertible matrix $X \in {\rm M}_n(K)$ such that $X^{-1}\mathcal{A}X$ is a subalgebra of the full subalgebra of M$_n(K)$ of type $(n_1, n_2, \ldots, n_q)$.
Moreover, the ideal $X^{-1}IX$ is contained in 
$$
\left(
\begin{array}{ccccc}
0_{n_1} & {\rm M}_{n_1\times n_2}(K) & {\rm M}_{n_1\times n_3}(K) & \cdots      & {\rm M}_{n_1\times n_q}(K)\\ 
       \\
        & 0_{n_2}              & {\rm M}_{n_2\times n_3}(K) & \cdots      & {\rm M}_{n_2\times n_q}(K)  \\  \\
        \\
        &                      & \ddots               & \ddots      & \vdots                \\ 
        \\
        &                      &                      & \ddots      & {\rm M}_{n_{q-1}\times n_q}(K)\\  
        \\
        &                      &                      &             & 0_{n_q}           
\end{array}
\right).
$$ 
\end{lemma}

\begin{proof} Denote the vector space $K^n$ by $V$. For $i= 1,2, \ldots, q$, let $n_i = {\rm dim}_KI^{q-i}V/I^{q-i+1}V$, with $I^{0}V :=V$.

By definition, $n_1 = {\rm dim}_K I^{q-1}V/I^qV= {\rm dim}_K I^{q-1}V$, since $I^q = 0$ (and hence $I^qV = 0$). Next, $I^{q-1}V$ is a $K$-subspace of $I^{q-2}V$, and so $n_1+n_2 = {\rm dim}_K I^{q-1}V+{\rm dim}_K I^{q-2}V/I^{q-1}V = {\rm dim}_K I^{q-2}V$. Inductively, assume that $n_1+n_2+\cdots + n_i = {\rm dim}_K I^{q-i}V$ for some positive integer $i < q$. Since $I^{q-i}V$ is a $K$-subspace of $I^{q-i-1}V$, we  conclude that 
\begin{align*}
n_1+n_2+\cdots+n_{i+1} &= {\rm dim}_K I^{q-i}V + n_{i+1} \\
                       &={\rm dim}_K I^{q-i}V + {\rm dim}_K I^{q-i-1}V/I^{q-i}V = {\rm dim}_K I^{q-i-1}V.
\end{align*}
Hence $n_1+n_2+\cdots + n_i = {\rm dim}_K I^{q-i}V$ for $i = 1, 2, \ldots, q$; in particular, $n_1+n_2+ \cdots+n_q = {\rm dim}_K V = {\rm dim}_K K^n = n.$ 

We have the following sequence of $K$-subspaces of $V$:
\begin{equation} \label{eq2}
0 \subseteq I^{q-1}V \subseteq I^{q-2}V \subseteq \cdots \subseteq IV \subseteq V.
\end{equation}
Using the assumption that $I^q = 0$ and $I^{q-1} \neq 0$, we will show that all the inclusions in~(\ref{eq2}) are proper. Suppose that
 $I^{q-1}V = 0$ or $I^{q-j}V = I^{q-j-1}V$ for some $j$, $1 \leq j \leq q-1$.
 Then 
 $$
 0= I^qV=I^j(I^{q-j}V)=I^j(I^{q-j-1}V)=I^{q-1}V,
 $$
 and so from $V = K^n$ we conclude that $I^{q-1} =0$ (otherwise some matrix in $I^{q-1}$ would have a nonzero entry in some row, which would in turn imply that $I^{q-1}V\ne0$); a contradiction. This establishes the proper inclusions. Thus, 
$$ 
1 \leq {\rm dim}_K I^{q-1}V, \ 2 \leq {\rm dim}_K I^{q-2}V,\,\cdots,\ q-1 \leq {\rm dim}_K IV,\ q \leq {\rm dim}_K V = n.
$$

Now, using the sequence of $K$-subspaces in (\ref{eq2}), we define a basis $B = (v_1, v_2, \ldots ,v_n)$ for the $K$-space $V$ in the following way:

Start with a basis $(v_1, v_2, \ldots, v_{n_1})$ for the $K$-space $I^{q-1}V$ (keeping in mind that, by definition, ${\rm dim}_K I^{q-1}V = n_1$). Next, $I^{q-1}V$ is a $K$-subspace of $I^{q-2}V$, and ${\rm dim}_K I^{q-2}V = n_1+n_2$. So, take vectors $v_{n_1+1}, v_{n_1+2}, \ldots, v_{n_1+n_2}$ such that $(v_1, v_2, \ldots, v_{n_1+n_2})$ is a basis 
for~$I^{q-2}V$. 
Continuing in this way, we construct a basis $B = (v_1, v_2, \ldots, v_n)$ for the $K$-space~$V$, where $(v_1, v_2, \ldots, v_{n_1+n_2+\cdots+n_i})$ is a basis for the $K$-space $I^{q-i}V, \ i =1, 2, \ldots, q-1$.

Now we take an arbitrary matrix $Y \in \A$. Let $\varphi \colon V \to V$ be a linear map such that the (transformation) matrix of  $\varphi$ with respect to the standard basis $E: = (e_1, e_2, \ldots, e_n)$ for $V$  is $M(\varphi)^E_E =Y$. (Note also that we have vectors $v_1$, $v_2$, \ldots, $v_n$ written in terms of E). As $I$ is an ideal of the algebra $\A$, we have $Y (I^{q-1}V) \subseteq \A (I^{q-1}V) \subseteq I^{q-1}V$. Since $(v_1, v_2, \ldots, v_{n_1})$ is a basis for $I^{q-1}V$, for any $i =1, 2, \ldots, n_1$ we can write 
$$
\varphi(v_i) = Yv_i = y_{1i}v_1+ y_{2i}v_2 + \cdots + y_{n_1,i}v_{n_i},
$$
for some scalars $y_{ji} \in K$ with $1\leq i, j \leq n_1$.
Similarly, $Y (I^{q-2}V) \subseteq \A (I^{q-2}V) \subseteq I^{q-2}V$, and $(v_1, v_2, \ldots, v_{n_1+n_2})$ is a basis of the $K$-space $I^{q-2}V$, and so for $j = n_1+1, n_1+2, \ldots, n_1+n_2$ we can write
$$
\varphi(v_j) = Yv_j = y_{1j}v_1+ y_{2j}v_2 + \ldots + y_{n_1+n_2,j}v_{n_1+n_2},
$$
for some scalars $y_{kj} \in K$ with $1 \leq k \leq n_1+n_2$ (and $n_1+1 \leq j \leq n_1+n_2)$. Continuing in this way, we eventually simply have that $Y V \subseteq V$, and with $(v_1, v_2, \ldots, v_n)$ being a basis for $V$, we can therefore, for $l = n_1+n_2+\cdots+n_{q-1}+1, n_1+n_2+\cdots+n_{q-1}+2, \ldots, n$, write 
$$
\varphi(v_l) = Yv_l = y_{1l}v_1+ y_{2l}v_2 + \ldots + y_{nl}v_n,
$$
for some scalars $y_{ml}$, with $1 \leq m \leq n$ (and $n_1+n_2+\cdots+ n_{q-1}+1 \leq l \leq n)$. Hence, using the notation 
$$
N_1:= n_1 \qquad {\rm and} \qquad N_i:=n_1 +\cdots+n_i, \ i=2,\ldots,q,
$$
which implies that $N_q=n$, the matrix~$M_B^B(\varphi)$ of the linear map $\varphi$ with respect to the basis~$B$  is the following:
$$
\left(
\begin{array}{cccccccccc}
y_{11}    & \ldots & y_{1,n_1}    & y_{1,n_1+1}           & \ldots  & y_{1,N_2}         & \ldots & y_{1,N_{q-1}+1}                    & \ldots &  y_{1,N_q}                  \\ 
\vdots    & \ddots & \vdots       & \vdots                & \ddots  & \vdots                & \ldots & \vdots                                        & \ddots &  \vdots                                \\
y_{n_1, 1}  & \ldots & y_{n_1, n_1} & y_{n_1, n_1+1}        &  \cdots & y_{n_1, N_2}      & \ldots & y_{n_1,N_{q-1}+1}                  & \ldots &  y_{n_1,N_q}                \\ 
       &        &              & y_{n_1+1, n_1+1}      &  \cdots & y_{n_1+1, N_2}    & \ldots & y_{n_1+1,N_{q-1}+1}                & \ldots &  y_{n_1+1,N_q}               \\
      &        &              & \vdots                &  \ddots & \vdots                & \ldots & \vdots                                        & \ddots &  \vdots                                  \\
        &       &              &  y_{N_2, n_1+1}   &  \ldots & y_{N_2, N_2}  & \ldots & y_{N_2,N_{q-1}+1}              & \ldots &  y_{N_2,N_q}               \\
       &      &              &                       &         &                       & \ddots &                      \vdots                         & \vdots &          \vdots                               \\
      &        &              &                       &         &                       &        & y_{N_{q-1}+1,N_{q-1}+1} & \ldots &  y_{N_{q-1}+1,N_q} \\
        &        &              &                &   &                & & \vdots                                        & \ddots &  \vdots                                  \\
       &       &              &                       &   &                       &  & y_{N_q,N_{q-1}+1}       & \ldots &  y_{N_q,N_q}      
\end{array}
\right).
$$

Consequently, $M_B^B(\varphi)$ is an element of the full subalgebra of M$_n(K)$ of type $(n_1, n_2, \ldots, n_q)$.
From the change-of-basis formula 
$$M_B^B(\varphi) = M(id)^B_E \cdot M(\varphi)^E_E \cdot M(id)_B^E = M(id)^B_E \cdot Y \cdot M(id)_B^E,$$ 
where the change-of-basis matrix $M(id)_B^E$ is the matrix with vector $v_j$ written in the $j$-th column, $j=1,2,\ldots,n,$ 
and $M(id)^B_E  = (M(id)^E_B)^{-1}$. Since $Y$ was an arbitrary matrix of  $\A$, we have proved that the algebra $X^{-1}\A X$ is in block triangular form, with $X = M(id)_B^E$. 

In order to complete the proof, it remains to show that matrices from $X^{-1}IX$ have
diagonal blocks only with zeros. Take an arbitrary matrix $Z \in I$. Let $\psi \colon V \to V$ be a linear map such that the (transformation) matrix of  $\psi$ with respect to the standard basis $E$ for~$V$  is $M(\psi)^E_E =Z$.
Since $I^q = 0$, it follows that $Z(I^{q-1}V) \subseteq I(I^{q-1}V) = 0$, and so, with $(v_1, v_2, \ldots, v_{n_1})$ being a basis for $I^{q-1}V$, we have, for $i = 1, 2, \ldots, n_1,$ 
$$\psi(v_i) = Z v_i = 0.$$ 
Next, $Z (I^{q-2}V) \subseteq I (I^{q-2}V) = I^{q-1}V$. Then for $j = n_1+1, n_1+2, \ldots, n_1+n_2 \ (=N_2)$ we can write
$$
\psi(v_j) = Z v_j = z_{1j}v_1 + z_{2j} v_2 + \ldots + z_{n_1,j}v_{n_1}
$$
for some scalars $z_{kj} \in K$ with $1 \leq k \leq n_1$ (and $n_1+1 \leq j \leq n_1+n_2$), because the basis $(v_1, v_2, \ldots, v_{n_1})$ for $I^{q-1}V$ was expanded to the basis $(v_1, \ldots, v_{n_1}, v_{n_1+1}, \ldots, v_{n_1+n_2})$ 
for~$I^{q-2}V$.
The pattern should now be clear from the arguments above, which leads us to concluding that the matrix~$M_B^B(\psi)$ of the linear map $\psi$ with respect to the basis~$B$  is the following:
$$
\left(
\begin{array}{ccccccccccccc}
\!\!\! 0 \!\!\!  & \ldots\!\!\! & 0 \!\!\!   & z_{1,n_1+1} \!\!\!          & \ldots \!\!\! & z_{1,N_2}  \!\!\!     & z_{1,N_2+1}  \!\!\!         & \ldots \!\!\! & z_{1,N_3}\!\!\!  & \ldots\!\!\! & z_{1,N_{q-1}+1}   \!\!\!                 & \ldots\!\!\! &  z_{1,N_q} \!\!\!  \!\!\!   \!\!\!            \\ 
\!\!\! \vdots \!\!\!   & \ddots\!\!\! & \vdots  \!\!\!     & \vdots  \!\!\!              & \ddots\!\!\!  & \vdots \!\!\!      & \vdots  \!\!\!              & \ddots \!\!\! & \vdots  \!\!\!         & \ldots\!\!\! & \vdots    \!\!\!                                    & \ddots\!\!\! &  \vdots            \!\!\!      \!\!\!              \\
\!\!\! 0 \!\!\! & \ldots\!\!\! & 0\!\!\! & z_{n_1, n_1+1} \!\!\!       &  \cdots\!\!\! & z_{n_1, N_2}\!\!\!  & z_{n_1, N_2+1} \!\!\!       &  \cdots\!\!\! & z_{n_1, N_3} \!\!\!   & \ldots\!\!\! & z_{n_1,N_{q-1}+1}       \!\!\!           & \ldots\!\!\! &  z_{n_1,N_q} \!\!\!  \!\!\!             \\ 
 \!\!\!       &   \!\!\!     &     \!\!\!         & 0  \!\!\!    &  \cdots\!\!\! & 0 \!\!\! & z_{n_1+1,N_2+1}   \!\!\!        & \ldots\!\!\!  & z_{n_1+1,N_3} \!\!\!  & \ldots\!\!\! & z_{n_1+1,N_{q-1}+1} \!\!\!               & \ldots\!\!\! &  z_{n_1+1,N_q} \!\!\!  \!\!\!   \!\!\!         \\
  \!\!\!     &   \!\!\!     &   \!\!\!           & \vdots     \!\!\!           &  \ddots\!\!\! & \vdots \!\!\!    & \vdots \!\!\!               &  \ddots \!\!\!& \vdots  \!\!\!          & \ldots\!\!\! & \vdots                   \!\!\!                     & \ddots \!\!\! &  \vdots                        \!\!\!   \!\!\!       \\
  \!\!\!       &  \!\!\!     &   \!\!\!           &  0 \!\!\!  &  \ldots\!\!\! & 0 \!\!\! & z_{N_2,N_2+1} \!\!\!          & \ldots \!\!\! & z_{N_2,N_3}\!\!\! & \ldots\!\!\! & z_{N_2,N_{q-1}+1}\!\!\!              & \ldots\!\!\! &  z_{N_2,N_q}       \!\!\!   \!\!\!     \\
 \!\!\!      &  \!\!\!    &      \!\!\!        &       \!\!\!                &    \!\!\!     &   \!\!\!           & \ddots\!\!\!        &\!\!\!  & \!\!\!    &                 \cdots  \!\!\!                       & \vdots \!\!\! &          \vdots\!\!\! &\vdots \!\!\! \!\!\!                             \\  
   \!\!\!      &  \!\!\!    &      \!\!\!        &       \!\!\!                &    \!\!\!     &   \!\!\!           & \!\!\!        & \ddots\!\!\! & \!\!\!    &                 \cdots  \!\!\!                       & \vdots \!\!\! &          \vdots\!\!\! &\vdots \!\!\! \!\!\!                             \\
  \!\!\!      &  \!\!\!    &      \!\!\!        &       \!\!\!                &    \!\!\!     &   \!\!\!           & \!\!\!        &\!\!\! &  \ddots\!\!\!    &                 \cdots  \!\!\!                       & \vdots \!\!\! &          \vdots\!\!\! &\vdots \!\!\! \!\!\!                             \\  
\!\!\!       &  \!\!\!     &    \!\!\!          &   \!\!\!      &\!\!\! &\!\!\! &\!\!\!              &\!\!\!   &                       &\!\!\!  & z_{N_{q-2}+1,N_{q-1}+1}  \!\!\!      & \ldots \!\!\!& z_{N_{q-2}+1,N_q}   \!\!\! \!\!\!   \\ 
 \!\!\!       &  \!\!\!     &   \!\!\!           &  \!\!\!       & \!\!\! &\!\!\! & \!\!\!             & \!\!\!  &   \!\!\!                    &\!\!\!  & \vdots \!\!\!       & \ddots\!\!\! & \vdots  \!\!\!  \!\!\!   \\    
 \!\!\!       &   \!\!\!    &  \!\!\!            &   \!\!\!      &\!\!\! &\!\!\! &  \!\!\!            & \!\!\!  &  \!\!\!                     & \!\!\! & z_{N_{q-1},N_{q-1}+1}\!\!\!        & \ldots\!\!\! & z_{N_{q-1},N_q}  \!\!\!  \!\!\!   \\        
 \!\!\!      &    \!\!\!    &    \!\!\!          &   \!\!\!      &\!\!\! &\!\!\! &   \!\!\!           & \!\!\!        &    \!\!\!                   & \!\!\!       & 0\!\!\! & \ldots\!\!\! &  0 \!\!\!\!\!\! \\
  \!\!\!       &  \!\!\!      &  \!\!\!            &  \!\!\!     &\!\!\! &\!\!\! &\!\!\!         &\!\!\!   & \!\!\!               &\!\!\! & \vdots  \!\!\!                                      & \ddots\!\!\! &  \vdots      \!\!\!   \!\!\!                         \\
  \!\!\!      &   \!\!\!    &    \!\!\!          &   \!\!\!      &\!\!\! &\!\!\! & \!\!\!             & \!\!\!  & \!\!\!                      & \!\!\! & 0 \!\!\!      & \ldots\!\!\! &  0      \!\!\!\!\!\!
\end{array}
\right).
$$
Hence, $M_B^B(\psi)$ is in the strictly upper block triangular part 
of the full subalgebra of M$_n(K)$ of type $(n_1, n_2, \ldots, n_q)$.
\end{proof}

\begin{remark} \label{numbers ni}
The numbers $n_1, n_2, \ldots, n_q$ of Lemma \ref{lemma block triangulation}, defined in the first line of the proof, are determined  by algebra $\mathcal{A}$ and  the dimensions of space $V$ and subspaces $I^{i} V$ for $i = 1, 2, \ldots, q-1$, where $V = K^n$. 
\end{remark}

Note that every finite dimensional $K$-algebra $\mathcal{A}$ 
can be identified with a subalgebra of M$_n(K)$, for $n \leq {\rm dim}_K\A$. To do this we can use, for example,  a regular representation. The Jacobson radical $J(\mathcal{A})$ of a finite dimensional algebra $\mathcal{A}$ is nilpotent (see \cite[Theorem 4.12]{Lam} for the broader class of Artinian rings), and so after such identification of $\mathcal{A}$ with a subalgebra of M$_n(K)$ we can find an algebra in block triangular form (as in the above lemma) which is a conjugated of $\mathcal{A}$. 

In the following example we start with the finite dimensional algebra $\mathcal{A}= {\rm M}_2(K[x]/(x^2))$. After identification with a subalgebra of matrices, the algebra $\mathcal{A}$ is conjugated with subalgebra of a block triangular matrices. We will describe the obtained blocks and see some ``dependence'' between them in the sense of the definition below formula (\ref{blocks2}). By $\overline{x}$ we will denote the image of $x \in K[x]$ in the natural homomorphism to the quotient algebra $K[x]/(x^2)$.

\begin{example}\label{example block triangular}
Let $\mathcal{A}$ be the finite dimensional algebra M$_2(K[x]/(x^2))$. Since, for any natural number $n$, the Jacobson radical satisfies $J({\rm M}_n(\mathcal{A})) = {\rm M}_n(J(\mathcal{A}))$ (see \cite[point (7), page 57]{Lam} ), we have $J(\mathcal{A}) = {\rm M}_2(J(K[x]/(x^2)) = {\rm M}_2(K\overline{x})$, which implies that $J(\mathcal{A})$ is a nonzero ideal with $(J(\mathcal{A}))^2 = 0$. Using the identification of an arbitrary element $a+b \overline{x}$ in $K[x]/(x^2)$ with the matrix
$
\left(
\begin{array}{cc}
a & b \\
0 & a
\end{array}
\right) \in {\rm M}_2(K)
$
we will treat (the 8-dimensional) $K$-algebra $\mathcal{A}$ as the subalgebra of M$_4(K)$ comprising all matrices of the form 
\begin{equation} \label{aij}
\left(
\begin{array}{cccc}
a_{11} & b_{11} & a_{12} & b_{12} \\ 
0      & a_{11} &      0 & a_{12} \\
a_{21} & b_{21} & a_{22} & b_{22} \\ 
0      & a_{21} &      0 & a_{22} \\
\end{array}
\right),
\end{equation}
where $a_{ij}, b_{ij} \in K$ for $1 \leq i, j \leq 2$. With this identification, we have 
\begin{equation} \label{aij in radical}
J(\mathcal{A}) = \left \{
\left(
\begin{array}{cccc}
0      & b_{11} &      0 & b_{12} \\ 
0      & 0      &      0 & 0    \\
0      & b_{21} &    0 & b_{22} \\ 
0      &  0     &      0 & 0   \\
\end{array}
\right): \ b_{ij} \in K \right\}.
\end{equation}

Now we are ready to use Lemma \ref{lemma block triangulation}. With $2$ being the nilpotency index $q$ of~$J(\mathcal{A})$, and with $V = K^4$, we have $J(\mathcal{A})V = {\rm span}(e_1, e_3)$, and so, following the notation 
in~Lemma~\ref{lemma block triangulation}, we have $n_1 = \dim_K J(\mathcal{A})V = 2$ and $n_2 = \dim_K (J(\mathcal{A}))^{0}V/J(\mathcal{A})V = \dim_K V/J(\mathcal{A})V = 2$. 
By~Lemma~\ref{lemma block triangulation}, there exists an invertible matrix $X\in {\rm M}_4(K)$ such that $X^{-1}\mathcal{A}X$ is  
a subalgebra of the full subalgebra of M$_4(K)$ of type $(2, 2)$ which is
$ \left(
\begin{array}{cc}
{\rm M}_{2}(K) & {\rm M}_{2}(K)  \\ 
               & {\rm M}_{2}(K) 
\end{array}
\right)
$,
and such that the ideal $X^{-1}J(\mathcal{A})X$ is contained in the strictly upper block triangular part  
$\left(
\begin{array}{cc}
0_2      & {\rm M}_{2}(K)\\
         & 0_2
\end{array}
\right).
$ 

In order to find such an $X$ 
we follow the proof of~Lemma~\ref{lemma block triangulation}.
 We start the construction of a basis $B$ for $V$ by first  finding basis vectors for~$J(\mathcal{A})V$. As $J(\mathcal{A})V = {\rm span}(e_1, e_3)$, we take  $(e_1, e_3)$ as a basis for~$J(\mathcal{A})V$.  As $q =2$, the second step is the last step, and in it we expand the basis $(e_1, e_3)$ to a basis for~$V$, by using $e_2$ and $e_4$, i.e., we take $B$ as $(e_1, e_3, e_2, e_4)$. 
An arbitrary matrix 
$$
\left(
\begin{array}{cccc}
a_{11} & b_{11} & a_{12} & b_{12} \\ 
0      & a_{11} &      0 & a_{12} \\
a_{21} & b_{21} & a_{22} & b_{22} \\ 
0      & a_{21} &      0 & a_{22} \\
\end{array}
\right)
$$
from algebra $\mathcal{A}$, treated as a linear map in the canonical basis~ $(e_1, e_2, e_3, e_4)$,  has 
$$
\left(
\begin{array}{cccc}
a_{11} & a_{12} & b_{11} & b_{12} \\ 
a_{21} & a_{22} & b_{21} & b_{22} \\
0      &     0  & a_{11} & a_{12} \\ 
0      &     0  & a_{21} & a_{22} \\
\end{array}
\right)$$ as transformation matrix with respect to the basis $B$.
It is obtained by conjugation with the matrix 
$$ X = 
\left(
\begin{array}{cccc}
1      &     0 &  0     & 0 \\ 
0      &     0  & 1     & 0 \\
0      &     1  & 0     & 0 \\ 
0      &     0  & 0     & 1 \\
\end{array}
\right)
$$ of vectors~$e_1, e_3, e_2, e_4$ written in the first, second, third and fourth column, respectively. Consequently, 
every matrix $A \in X^{-1}\mathcal{A}X$ is in the block form
$
\left(
\begin{array}{cc}
A_{11} & A_{12} \\
       & A_{11}
\end{array}
\right)$,
where $A_{11}$ and $A_{12}$ are any matrices from M$_2(K)$. Importantly, the two matrices in the diagonal blocks are equal (denoted here by $A_{11}$).
Since every matrix in~$J(\mathcal{A})$ has entries $a_{ij} = 0$ (see $(\ref{aij})$ and $(\ref{aij in radical})$), we have $X^{-1}J(\mathcal{A})X \subseteq 
\left(
\begin{array}{cc}
0_2      & {\rm M}_2(K) \\
         & 0_2
\end{array}
\right).
$ 

Note that this example shows an interesting isomorphism, namely conjugation of the algebra M$_2({\rm U}^*_2(K))$ with the algebra U$^*_2( {\rm M}_2(K))$.
\end{example}



\medskip

For a D$_q$ algebra $\mathcal{A}$, we denote the ideal of $\mathcal{A}$ generated by the set $\{[x,y]: \ x, y \in \mathcal{A}\}$ of commutators in $\mathcal{A}$ by $\mathcal{C}_{\A}$.

\begin{proposition} \label{proposition identities with nilpotent ideal}
If $\mathcal{A}$ is a D$_q$ algebra, 
then $\mathcal{C}_{\A}^q = 0$ and $q$ is the nilpotency index of $\C_\A$. If, in addition, $\mathcal{A}$ is a subalgebra of~M$_n(K)$, then $q \leq n$. 
\end{proposition}
\begin{proof}
In order to show that the ideal $\mathcal{C}_{\A}$ is nilpotent with $\mathcal{C}_{\A}^q = 0$, take an element $x \in \mathcal{C}_{\A}$ of the following  form: 
$$x = r_1 \cdot [x_1,y_1] \cdot r_2 \cdot [x_2,y_2]  \cdot \ldots \cdot r_q \cdot [x_q,y_q] \cdot r_{q+1}. $$
Since $\mathcal{C}_{\A}^q$ comprises (finite) sum of elements of this form, it suffices to show that $x = 0$.

For any $a, b, r \in \mathcal{A}$ we have the identity $[a,rb] =[a,r]b+r[a,b]$, and so 
$$r[a,b] = [a,rb]-[a,r]b.$$
Applying the last equality to $x_1, y_1, r_1$,  we have
\begin{align*}
&r_1 \cdot [x_1,y_1]\cdot r_2 \cdot [x_2,y_2] \ldots \cdot r_q \cdot[x_q,y_q] \cdot r_{q+1} \\
&= ([x_1,r_1y_1]-[x_1,r_1]y_1) \cdot r_2 \cdot [x_2,y_2] \cdot \ldots \cdot r_q \cdot[x_q,y_q] \cdot r_{q+1} \\
&= [x_1,r_1y_1] \cdot r_2 \cdot [x_2,y_2] \cdot \ldots \cdot r_q \cdot[x_q,y_q] \cdot r_{q+1}\\
&-[x_1,r_1] \cdot y_1  r_2 \cdot [x_2,y_2] \cdot r_3 \cdot [x_3,y_3] \cdot \ldots \cdot r_q \cdot[x_q,y_q] \cdot r_{q+1}.
\end{align*}
Next we can write
\begin{align*}
& [x_1,r_1y_1] \cdot r_2 \cdot [x_2,y_2] \cdot r_3 \cdot [x_3,y_3] \cdot r_4 \ldots \cdot r_q \cdot[x_q,y_q] \cdot r_{q+1}  \\
&= [x_1,r_1y_1] \cdot ([x_2,r_2y_2]-[x_2,r_2]y_2) \cdot r_3 \cdot [x_3,y_3] \cdot r_4 \ldots \cdot r_q \cdot[x_q,y_q] \cdot r_{q+1} \\
&= [x_1,r_1y_1] [x_2,r_2y_2] \cdot r_3 \cdot [x_3,y_3] \cdot r_4 \cdot [x_4,y_4] \cdot \ldots \cdot r_q \cdot[x_q,y_q] \cdot r_{q+1} \\
&- [x_1,r_1y_1][x_2,r_2] \cdot y_2  r_3 \cdot [x_3,y_3] \cdot r_4 \cdot [x_4,y_4] \cdot \ldots \cdot r_q \cdot[x_q,y_q] \cdot r_{q+1}
\end{align*}
and
\begin{align*}
& [x_1,r_1] \cdot y_1  r_2 \cdot [x_2,y_2] \cdot r_3 \cdot [x_3,y_3] \cdot \ldots \cdot r_q \cdot[x_q,y_q] \cdot r_{q+1}  \\
&= [x_1,r_1] ([x_2,y_1 r_2 y_2]-[x_2,y_1r_2]y_2) \cdot r_3 \cdot [x_3,y_3] \cdot \ldots \cdot r_q \cdot[x_q,y_q] \cdot r_{q+1}  \\
&= [x_1,r_1] [x_2,y_1r_2y_2] \cdot r_3 \cdot [x_3,y_3] \cdot r_4 \cdot [x_4,y_4] \cdot \ldots \cdot r_q \cdot[x_q,y_q] \cdot r_{q+1} \\
&- [x_1,r_1] [x_2,y_1r_2] \cdot y_2 r_3 \cdot [x_3,y_3] \cdot r_4 \cdot [x_4,y_4] \cdot \ldots \cdot r_q \cdot[x_q,y_q] \cdot r_{q+1}.
\end{align*}
Continuing is this way, it is evident that $x$ can be written as a sum of elements
of the form~
$$\pm ( [x'_1,y'_1] [x'_2,y'_2] \ldots [x'_q,y'_q])r $$
for some $x'_1, y'_1, x'_2, y'_2, \ldots, x'_q, y'_q, r \in \mathcal{A}$.
Such elements are all equal to zero, because $\mathcal{A}$ is  a D$_{q}$ algebra. Hence, $\mathcal{C}_{\A}^q = 0$. 

Recall from the discussion preceding Theorem \ref{Domokos} that we always assume that $q>1$ and that D$_q$ algebra $\A$ is not a D$_{q-1}$ algebra. So $q$ is the nilpotency index of $\C_{\A}$.

If, in addition, $\mathcal{A}$ is a subalgebra of M$_n(K)$, 
then it follows 
from~Lemma~\ref{lemma block triangulation} that $q \leq n$. 
\end{proof}

Obviously, Proposition \ref{proposition identities with nilpotent ideal} implies the following fact:
\begin{corollary} \label{Dn+1 are Dn}
For every positive integer $n$ there are not $D_{q}$ subalgebras of M$_n(K)$ for every $q>n$.
\end{corollary}


\medskip

\section{Maximal D$_q$ subalgebras of M$_n(K)$ are conjugated with D$_q$ subalgebras with max-comm db's}\label{section3}  

In this section we will characterize, up to conjugation, 
maximal D$_q$ subalgebras of M$_n(K)$, in particular these with maximum dimension.

We will show in Theorem \ref{theorem block form for D_q} that every D$_q$ subalgebra of M$_n(K)$ is  conjugated with  a subalgebra of a full subalgebra of M$_n(K)$ of type $(n_1, n_2, \ldots, n_q)$, and it posseses some interesting additional properties. 

Using this result we prove in Theorem \ref{corollary block} that every
maximal D$_q$ 
subalgebra of M$_n(K)$ is  conjugated with  a {\rm D}$_q$ subalgebra of~{\rm M}$_n(K)$ of type $(n_1, n_2, \ldots, n_q)$ with max-comm db's.

 We then conclude in Corollary \ref{cormaintheorem} that every D$_q$ subalgebra of M$_n(K)$ with maximum dimension is conjugated with  max-dim {\rm D}$_q$ subalgebra of~{\rm M}$_n(K)$ of type $(n_1, n_2, \ldots, n_q)$ with max-dim db's.

\begin{theorem} \label{theorem block form for D_q}
Let $\mathcal{A}$ be a D$_q$ subalgebra of M$_n(K)$.
Then there exist positive integers $n_1, n_2, \ldots, n_q$, such that $\sum_{i=1}^q n_i = n$ and an invertible matrix $X \in {\rm M}_n(K)$, such that every matrix~$A'$ in the algebra~$\mathcal{A}' = X^{-1}\mathcal{A}X$ is in block triangular form
\begin{equation} \label{blok triangular 2}
\left(
\begin{array}{cccc}
A'_{11} & A'_{12}  & \ldots      & A'_{1q} \\ 
        & A'_{22}  & \ldots      & A'_{2q}  \\ 
        &         & \ddots      & \vdots   \\ 
        &         &             &   A'_{qq}            
\end{array}
\right),
\end{equation}
where  $A'_{ij} \in {\rm M}_{n_i \times n_j}(K)$
for all $i$ and $j$ such that $1 \leq i \leq j \leq q$ (and 
other entries are zero) and $\overline{\mathcal{A'}}_{ii}$, defined in (\ref{block}), is a commutative subalgebra of M$_{n_i}(K)$ for every $i, \ i =1,2 \ldots, q$. 
\end{theorem}

\begin{proof}


By Proposition \ref{proposition identities with nilpotent ideal}, $\mathcal{C}_{\A}^q = 0$, where $\mathcal{C}_{\A}$ is the ideal of $\mathcal{A}$ generated by all commutators in~$\mathcal{A}$, and $q$ is the nilpotency index of $\mathcal{C}_{\A}$.
Recall from the discussion preceding Theorem \ref{Domokos} that we always assume that $q>1$. 
Thus, by Lemma \ref{lemma block triangulation}, there exists an invertible matrix $X \in {\rm M}_n(K)$ such that every matrix $A'$ in the algebra $\mathcal{A}' = X^{-1}\mathcal{A}X$ is in the block triangular form (\ref{blok triangular 2}), and 
the ideal~$X^{-1} \mathcal{C}_{\A} X$ of the algebra 
$\mathcal{A}'$
comprises zero matrices in the diagonal blocks.

It remains to show that, for $i = 1, 2, \ldots, q$, the subalgebra $\overline{\mathcal{A}'}_{ii}$ of M$_{n_i}(K)$ is commutative. Firstly, we will say more about the structure of the ideal $\C_{\A'}$ generated by all commutators $[z,w]$, $z,w \in \mathcal{A}'$.
Since conjugation is an isomorphism of algebras, it follows readily that the ideal $X^{-1} \mathcal{C}_{\A}X$ is equal to $\C_{\A'}$. 
Therefore 

$$
\C_{\A'} \subseteq
\left(
\begin{array}{ccccc}
0_{n_1} & {\rm M}_{n_1\times n_2}(K) & {\rm M}_{n_1\times n_3}(K) & \cdots      & {\rm M}_{n_1\times n_q}(K)\\ 
       \\
        & 0_{n_2}              & {\rm M}_{n_2\times n_3}(K) & \cdots      & {\rm M}_{n_2\times n_q}(K)  \\  \\
        \\
        &                      & \ddots               & \ddots      & \vdots                \\ 
        \\
        &                      &                      & \ddots      & {\rm M}_{n_{q-1}\times n_q}(K)\\  
        \\
        &                      &                      &             & 0_{n_q}           
\end{array}
\right).
$$

To complete the proof let $X_{ii}$, $Y_{ii}
\in\overline{\mathcal{A}'}_{ii}$. 
Then, by the definition,
there are block triangular matrices $X,Y\in\mathcal{A'}$ such that
$$
X = 
\left(
\begin{array}{cccccc}
X_{11} & \ldots  & X_{1i}  & \ldots   & X_{1q} \\ 
       & \ddots  & \vdots  & \ddots   & \vdots  \\ 
       &         & X_{ii}  & \ldots   & X_{iq} \\
       &         &         & \ddots   & \vdots \\
       &         &         &          & X_{qq}            
\end{array}
\right)
\quad {\rm and} \quad  
Y=
\left(
\begin{array}{cccccc}
Y_{11} & \ldots  & Y_{1i}  & \ldots   & Y_{1q} \\ 
       & \ddots  & \vdots  & \ddots   & \vdots  \\ 
       &         & Y_{ii}  & \ldots   & Y_{iq} \\
       &         &         & \ddots   & \vdots \\
       &         &         &          & Y_{qq}            
\end{array}
\right). 
$$  
The commutator $[X,Y]$ has the matrix $[X_{ii}, Y_{ii}]$ in the $i$-th diagonal block. Since we showed in the preceding paragraph that the diagonal blocks of the ideal generated by all commutators of $\mathcal{A}'$ are zero, we conclude that $[X_{ii}, Y_{ii}] = 0_{n_i}$, which completes the proof.
\end{proof}

Next we show that if, in addition, $\mathcal{A}$ is a maximal D$_q$ subalgebra of M$_n(K)$, then the obtained 
algebras $\overline{\mathcal{A'}}_{ii}$ above are independent (see the definition below formula (\ref{blocks2})).
To be precise, we have  the following characterization: 

\begin{theorem} \label{corollary block}
Let $\mathcal{A}$ be a maximal D$_q$ 
subalgebra of M$_n(K)$. 
Then there exists an invertible matrix~$X \in {\rm M}_n(K)$ such that $X^{-1}\mathcal{A}X$ is a 
{\rm D}$_q$ subalgebra of~{\rm M}$_n(K)$ of some type $(n_1, n_2, \ldots, n_q)$ with max-comm db's.
\end{theorem}

\begin{proof}

By Theorem \ref{theorem block form for D_q},
there exists an invertible matrix $X$ such that every matrix $A' \in \A' = X^{-1}\mathcal{A}X$ is in block triangular form
$$
\left(
\begin{array}{cccc}
A'_{11} & A'_{12}  & \ldots      & A'_{1q} \\ 
       & A'_{22}  & \ldots      & A'_{2q}  \\ 
       &         & \ddots      & \vdots   \\ 
       &         &             & A'_{qq}            
\end{array}
\right),
$$
where $A'_{ij} \in {\rm M}_{n_i \times n_j}(K)$ for $1 \leq i \leq j \leq q$ and each $\overline{\A'}_{ii}$ is a commutative subalgebra of M$_{n_i}(K)$. 


Let $\B$ be the subalgebra of M$_n(K)$ of type $(n_1, n_2, \ldots, n_q)$ equal to

$$
\left(
\begin{array}{ccccc}
\overline{\A'}_{11} & {\rm M}_{n_1\times n_2}(K) & {\rm M}_{n_1\times n_3}(K) & \cdots      & {\rm M}_{n_1\times n_q}(K)\\ 
       \\
        & \overline{\A'}_{22}              & {\rm M}_{n_2\times n_3}(K) & \cdots      & {\rm M}_{n_2\times n_q}(K)  \\  \\
        \\
        &                      & \ddots               & \ddots      & \vdots                \\ 
        \\
        &                      &                      & \ddots      & {\rm M}_{n_{q-1}\times n_q}(K)\\  
        \\
        &                      &                      &             & \overline{\A'}_{qq}          
\end{array}
\right).
$$
Note that by Remark \ref{typicaldq} $\B$ is a D$_q$ subalgebra of M$_n(K)$. Since $\A$ is a maximal D$_q$ subalgebra of~M$_n(K)$, it follows that $\A' = X^{-1}\A X$ 
is also a maximal D$_q$ subalgebra of~M$_n(K)$. So from the inclusion $\A' \subseteq \B$ we obtain the equality $\A' = \B$.
 
In order to complete the proof we will show that each $\overline{\A}_{ii}$ is a maximal commutative subalgebra of~M$_{n_i}(K)$. Suppose, for the contrary, that, for some $j \in \{1, 2, \ldots, q \}$, the diagonal block $\overline{\A}_{jj}$ is properly contained in a commutative  subalgebra $C_{jj}$ of M$_{n_j}(K)$. Then changing $\overline{\A}_{jj}$ to $C_{jj}$ produces a  D$_q$ subalgebra of M$_n(K)$ properly containing $\A'$, a contradiction. It completes the proof.  
\end{proof}

\begin{remark} \label{maxtuple}

Similar to Remark \ref{numbers ni}, the $q$-tuple $(n_1, n_2, \ldots, n_q)$ obtained in the proof of Theorem~\ref{corollary block} is determined by the dimensions of the vector space $V$ and the subspaces $\C_\A^i V$ for $i =1, 2, \ldots, q$, where $V = K^n$ and $\C_\A$ is the ideal generated by all commutators of the maximal D$_q$ subalgebra $\A$ of~M$_n(K)$. 

We will show in Theorem \ref{max up to conj} that two D$_q$ subalgebras of types $(n_1, n_2, \ldots, n_q)$ and $(\ell_1, \ell_2, \ldots, \ell_q)$ with max-comm db's are conjugated if and only if  $(n_1, n_2, \ldots, n_q)=(\ell_1, \ell_2, \ldots, \ell_q)$ (i.e., the $q$-tuple is uniquely determined) and the diagonal blocks of the D$_q$ algebras are pairwise conjugated. 

In summary, with an arbitrary maximal D$_q$ subalgebra $\A$ of M$_n(K)$ we can associate exactly one tuple $(n_1, n_2, \ldots, n_q)$ such that $\A$ is conjugated with a D$_q$ subalgebra of M$_n(K)$ of type $(n_1, n_2, \ldots, n_q)$ with max-comm db's.  

\end{remark}

If $\A$ is a D$_q$ subalgebra of M$_n(K)$ with maximum dimension, then by Theorem \ref{corollary block}, $\A$ is conjugated
with a D$_q$ subalgebra $\A'$ of M$_n(K)$ of some type $(n_1, n_2, \dots , n_q)$ with max-comm db’s. As in the
proof of Theorem \ref{corollary block}, if one of the diagonal blocks $\A'_{jj}$ of $\A'$ is not a commutative subalgebra of M$_{n_j}(K)$
with maximum dimension, then we can change this block and obtain a D$_q$ subalgebra with dimension greater
than that of $\A'$. This contradiction yields to following result:

\begin{corollary}\label{cormaintheorem}
Let $\mathcal{A}$ be a D$_q$ subalgebra of M$_n(K)$ with maximum dimension. Then there exists an invertible matrix $X$ such that $X^{-1}\mathcal{A}X$ is a 
max-dim {\rm D}$_q$ subalgebra of~{\rm M}$_n(K)$ of some type $(n_1, n_2, \ldots, n_q)$ with max-dim db's.
\end{corollary}


\section{Structure of D$_q$ subalgebras with max-comm db's}\label{section 4}
In Section \ref{section3} (see Theorem \ref{corollary block}) we showed that every maximal D$_q$ subalgebra of M$_n(K)$ is conjugated to a {\rm D}$_q$ subalgebra of~{\rm M}$_n(K)$ of type $(n_1, n_2, \ldots, n_q)$ with max-comm db's. In the present section, in Theorem \ref{typical are maximal}, we will prove that the converse is also true. 

Next, we will further analyze conjugations of {\rm D}$_q$ subalgebras of~M$_n(K)$.
In Proposition \ref{conjugation block triangular}, we will establish that conjugation, which satisfies some additional properties, of a 
D$_q$~subalgebra of~{\rm M}$_n(K)$ of any type $(n_1, n_2, \ldots, n_q)$ with max-comm db's  
is also a D$_q$ subalgebra of~{\rm M}$_n(K)$ of the same type with max-comm db's. 
A consequence is Corollary \ref{upper triangular max}, in which we will show that if $\mathcal{A}$ is a maximal D$_q$ subalgebra 
of~M$_n(K)$ contained in~U$_n(K)$, then  $\mathcal{A}$ is a D$_q$ subalgebra of M$_n(K)$ of some 
type~$(n_1,n_2,\ldots, n_q)$ with max-comm db's. It leads us to a negative answer to Question 9 posed in \cite{LM2}. Using 
Definition~\ref{abbreviations} and the paragraph immediately following it, we can rephrase the mentioned question as follows:    

\begin{question}\cite[Question 9]{LM2} \label{question typical}
For a field $K$, is there, for some $n$, a D$_2$ $K$-subalgebra of the upper triangular matrix algebra U$_n(K)$ with maximum dimension $2+\left\lfloor\frac{3n^2}{8}\right\rfloor$ which is not a max-dim D$_2$~subalgebra of M$_n(K)$ of some type $(n_1,n_2)$ with max-dim db's?
\end{question}    

In the same paper (see \cite[Theorem 15]{LM2}), a block triangular structure as in max-dim D$_q$ subalgebras of M$_n(K)$ with max-dim db's was proven for D$_2$ subalgebras of M$_n(K)$ with maximum dimension which are contained in U$_n(K)$ and satisfy some additional conditions. Corollary \ref{upper triangular max} generalizes this result.

Moreover, from Proposition \ref{conjugation block triangular} we obtain Theorem \ref{max up to conj}, which says that any D$_q$ subalgebras $\A$ and $\B$ of M$_n(K)$ 
with max-comm db's $\A_{ii}$ and $\B_{ii}$, respectively, $i =1, 2, \ldots, q$,  are conjugated if and only if they are of the same type and for each $i, \ i =1, 2, \ldots, q$,\  $\A_{ii}$ and $\B_{ii}$ are conjugates. It reduces the conjugation problem of maximal D$_q$ subalgebras of M$_n(K)$ to the conjugation problem of commutative subalgebras of M$_\ell(K)$, for $\ell = 1, 2, \ldots, n-1$. 
We will discuss the obtained result in Section \ref{remarks on jacobson}, restricting our attention to algebraically closed fields.

Recalling Proposition \ref{proposition identities with nilpotent ideal}, the first result in the present section describes powers of the ideal~$\mathcal{C}_{\A}$ generated by all commutators of a  
D$_q$ subalgebra $\A$ of M$_n(K)$ of type $(n_1, n_2, \ldots, n_q)$ with max-comm~db's.

\begin{proposition} \label{proposition power of ideal generated by commutators}
If $\A$ is D$_q$ subalgebra of~{\rm M}$_n(K)$ of type $(n_1, n_2, \ldots, n_q)$ with max-comm db's, then, for $i=1,2,\ldots,q-1,$

$$
\mathcal{C}_{\A}^i=\left(
\begin{array}{ccrccc} 
0_{n_1}  \ \ \ \ & \ldots  \ \  \ \  & 0_{n_1 \times n_{i}}  & \ \ \ {\rm M}_{n_1 \times n_{i+1}}(K) &  \ \ \ \ldots  & \ \ \ {\rm M}_{n_1 \times n_q}(K)      \\
\\ \\
        &  \ddots    &  \ddots               &   \ddots                  &  \ddots & \vdots                     \\
\\ \\
  		&            &  \ddots               &   \ddots                  & \ddots  & {\rm M}_{n_{q-i} \times n_q}(K)  \\
  		\\
        &            &                       &    \ddots                 &  \ddots &  0_{n_{q-i+1} \times n_q}  \\
        \\
        &            &                       &                           &  \ddots & \vdots                     \\
        \\
        &            &                       &                           &         & 0_{n_q}          
\end{array}
\right).
$$
\end{proposition}

\begin{proof}
We start with $i = 1$. Keeping in mind that the diagonal blocks of  
a D$_q$ subalgebra of~{\rm M}$_n(K)$ of type~$(n_1, n_2, \ldots, n_q)$ with max-comm db's 
are commutative, the inclusion 
$$ \mathcal{C}_{\mathcal{A}} \subseteq 
\left(
\begin{array}{ccccc}
0_{n_1} & {\rm M}_{n_1\times n_2}(K) & {\rm M}_{n_1\times n_3}(K) & \cdots      & {\rm M}_{n_1\times n_q}(K)\\ 
       \\
        & 0_{n_2}              & {\rm M}_{n_2\times n_3}(K) & \cdots      & {\rm M}_{n_2\times n_q}(K)  \\  \\
        \\
        &                      & \ddots               & \ddots      & \vdots                \\ 
        \\
        &                      &                      & \ddots      & {\rm M}_{n_{q-1}\times n_q}(K)\\  
        \\
        &                      &                      &             & 0_{n_q}           
\end{array}
\right)
$$ 
is immediate. 

In order to show the converse inclusion (for $i=1$), let $j$ be any positive integer such that $j<q$, and take arbitrary matrices $X_{j,j+1} \in {\rm M}_{n_j \times n_{j+1}}(K), \ X_{j, j+2} \in {\rm M}_{n_j \times n_{j+2}}(K), \ \ldots \ , \ X_{jq} \in {\rm M}_{n_j \times n_q}(K)$. Then
$$\left( 
\begin{array}{c|c|c}
 &        & \\ \hline
 & I_{n_j} & \\ \hline 
 &        &
\end{array}
\right) \qquad {\rm and} \qquad 
\left( 
\begin{array}{c|c|ccc}
 &          &           &        &    \\ \hline
 & 0_{n_j}  & X_{j, j+1} & \ldots & X_{jq} \\ \hline 
 &          &           &         &
\end{array}
\right)
$$
are elements of every 
D$_q$ subalgebra of M$_n(K)$ of type $(n_1, n_2, \ldots, n_q)$ with max-comm db's 
and so, since 
\begin{align*}
&\left( 
\begin{array}{c|c|c}
 &        & \\ \hline
 & I_{n_j} & \\ \hline 
 &        &
\end{array}
\right)    
\left( 
\begin{array}{c|c|ccc}
 &          &           &        &    \\ \hline
 & 0_{n_j}  & X_{j, j+1} & \ldots & X_{jq} \\ \hline 
 &          &           &         &
\end{array}
\right)=
\left( 
\begin{array}{c|c|ccc}
 &    &     &        &    \\ \hline
 & 0_{n_j}  & X_{j, j+1} & \ldots & X_{jq} \\ \hline 
 &          &          &         &
\end{array}
\right), \\
&\left( 
\begin{array}{c|c|ccc}
 &    &     &        &    \\ \hline
 & 0_{n_j}  & X_{j, j+1} & \ldots & X_{jq} \\ \hline 
 &          &          &         &
\end{array}
\right)
\left( 
\begin{array}{c|c|c}
 &        & \\ \hline
 & I_{n_j} & \\ \hline 
 &        &
\end{array}
\right)  
= 0_n,
\end{align*}
it follows that 
$$\left(
\begin{array}{c|c|ccc}
 &    &           &        &    \\ \hline
 & 0_{n_j}  & X_{j, j+1} & \ldots & X_{jq} \\ \hline 
 &    &          &         &
\end{array}
\right)=
\left[
\left( 
\begin{array}{c|c|c}
 &        & \\ \hline
 & I_{n_j} & \\ \hline 
 &        &
\end{array}
\right),
\left(
\begin{array}{c|c|ccc}
 &    &           &        &    \\ \hline
 & 0_{n_j}  & X_{j, j+1} & \ldots & X_{jq} \\ \hline 
 &    &          &         &
\end{array}
\right)
\right] \in\mathcal{C}_{\A}.
$$
As $j$ and the matrices $X_{j, j+1}$, $X_{j, j+2}$, \ldots, $X_{jq}$ were arbitrary, the mentioned inclusion has been established.

The form of the ideal $\mathcal{C}_{\A}^{i}$, for $i=2,\ldots,q-1$, is now evident.
\end{proof}

With the help of Proposition \ref{proposition power of ideal generated by commutators} 
  we are ready to prove the 
first of the main results of this section.
Henceforth, $e_{ij}$ denotes the matrix unit which has 1 in position $(i, j)$ and zeroes elsewhere. 

\begin{theorem} \label{typical are maximal}
Let $\mathcal{A}$ be a D$_q$
subalgebra of~{\rm M}$_n(K)$ of type $(n_1, n_2, \ldots, n_q)$ with max-comm db's. 
Then $\mathcal{A}$ is a maximal D$_q$ subalgebra of M$_n(K)$.
\end{theorem}

\begin{proof}
Suppose, for the contrary, that a D$_q$ subalgebra $\A$ of type $(n_1, n_2, \ldots, n_q)$ with max-comm db's  $\A_{ii}$, $i = 1, 2, \ldots, q$ (using the notation (\ref{block triang})), is not a maximal D$_q$ subalgebras of M$_n(K)$. The block structure (\ref{block triang}) of the algebra $\A$ will be 
essential to the proof. 

Let $\mathcal{B}$ be a maximal D$_q$ subalgebra of M$_n(K)$ properly containing $\mathcal{A}$. 
So we can find a matrix $X \in \mathcal{B} \setminus \A$. Write it in the block form

\begin{equation}\label{Xij}
\left(
\begin{array}{ccc}
X_{11} & \ldots & X_{1q} \\
\vdots & \ddots & \vdots \\
X_{q1}  & \ldots & X_{qq}
\end{array}
\right),
\end{equation}
where $X_{ij} \in {\rm M}_{n_i \times n_j}(K)$, $1 \leq i, j \leq q$. Note that the numbers $n_{i}$ are 
the same as those in the definition of the type of subalgebra $\A$.
Subalgebra $\A$ 
contains idempotents $E_1 = \sum_{i=1}^{n_1} e_{ii}$ and $E_{j} = \sum_{i = N_{j-1}+1}^{N_j} e_{ii}$ for $j =2, 3, \ldots, q$, where $N_{j} = n_1+n_2+\cdots+n_j$.
So they also belong to $\mathcal{B}$. It follows that for all indices $1 \leq i, j \leq q$, the matrices $E_i X E_j$ are in $\B$. 
These matrices in the form~$(\ref{Xij})$ have $X_{ij}$ in their $i$-th row and $j$-th column, and 
$0$ everywhere else. 
We conclude that there exists a matrix $Z$ in the ideal $\mathcal{C}_\mathcal{B}$ generated by all the commutators of $\mathcal{B}$, such that, written in the block form analogous to $(\ref{Xij})$, has exactly one nonzero block $Z_{rs}$, where $r$ and $s$ satisfy 
$1 \leq s \leq r \leq q$.

From the definition of $\A$ follows that $X \not \in \A$ if and only if either there exists $i>j$ such that $X_{ij}$ is a nonzero matrix or there exists $k \in \{1, 2, \ldots, q\}$ such that $X_{kk} \not \in \A_{kk}$. 

In the first case, for matrix $Z$ we can take $E_{ii} X E_{jj}$. This matrix belongs to $\C_\B$, because $E_{ii} X E_{jj} = [E_{ii}, E_{ii} X E_{jj}]$. 
In the other case, $X_{kk} \not \in \A_{kk}$. Since $\A_{kk}$ is a maximal commutative subalgebra, there exists a matrix $Y_{kk} \in \mathcal{A}_{kk}$ such that the commutator  $[X_{kk}, Y_{kk}]$ is nonzero. Let $Z$ be defined as follows:
\begin{align*}
Z =
\left[
\left(
\begin{array}{ccccc}
\ddots &             &         &             & \\
       & 0_{n_{k-1}} &         &             & \\
       &             & X_{kk}  &             & \\
       &             &         & 0_{n_{k+1}} & \\   
       &             &         &             &  \ddots
\end{array}
\right),
\left(
\begin{array}{ccccc}
\ddots &             &         &             & \\
       & 0_{n_{k-1}} &         &             & \\
       &             & Y_{kk} &              & \\
       &             &         & 0_{n_{k+1}} & \\   
       &             &         &             &  \ddots
\end{array}
\right)
\right].
\end{align*}
Then $Z \in \mathcal{C}_\mathcal{B}$, with the only nonzero $Z_{rs} \in {\rm M}_{n_r \times n_s}(K), 1\leq s \leq r \leq q$, in the form analogous to (\ref{Xij}), exists. Let $z_{ij}$, $1 \leq i \leq n_r$, $1 \leq j \leq n_s$, be a nonzero entry of the matrix $Z_{rs}$. Note that in matrix $Z$ it is entry $(N_{r-1}+i,N_{s-1}+j)$, where $N_{r-1} = n_1+n_2+\ldots+n_{r-1}$ and $N_{s-1} = n_1+n_2+\ldots+n_{s-1}$ (with $N_0 := 0$). Assume firstly that $r \not \in \{1,q \} $. Then 
$$
e_{1,N_{r-1}+i}\cdot Z \cdot e_{N_{s-1}+j,n} = z_{ij}e_{1n}\ne 0.
$$ 	
By Proposition \ref{proposition power of ideal generated by commutators},  $e_{1,N_{r-1}+i} \in (\mathcal{C}_\mathcal{A})^{r-1}$ and $e_{N_{s-1}+j,n} \in (\mathcal{C}_\mathcal{A})^{q-s}$, where $\mathcal{C}_\mathcal{A}$ is the ideal generated by all commutators $[x,y], \ x,y \in \mathcal{A}$. Since $\mathcal{B}$ contains $\A$,  we have that $\mathcal{C}_\mathcal{B}$ also contains $\mathcal{C}_\mathcal{A}$, and so
$$
z_{ij}e_{ij} \in (\mathcal{C}_\mathcal{B})^{r-1} \cdot \mathcal{C}_\mathcal{B} \cdot (\mathcal{C}_\mathcal{B})^{q-s}=(\mathcal{C}_\mathcal{B})^{q+(r-s)}= \{0\}.
$$
The above equality holds because $r-s\geq 0$ and by Proposition \ref{proposition identities with nilpotent ideal}, $\mathcal{C}_\mathcal{B}^q=\{0\}$. It is a contradiction, since $z_{ij}\ne 0$. 

When $r = 1$, then $Z_{11}$ is the nonzero block of matrix $Z$. By Proposition \ref{proposition power of ideal generated by commutators}, $e_{jn} \in \mathcal{C}_\mathcal{A}^{q-1} \subseteq \mathcal{C}_\mathcal{B}^{q-1}$, and so the product $Z \cdot e_{jn}$ in $\mathcal{C}_{\mathcal{B}} \cdot \mathcal{C}_\mathcal{B}^{q-1} = \mathcal{C}_\mathcal{B}^{q}$ is zero. 
However, $Z \cdot e_{jn}$ has as its $n$-th column the $j$-th column of $Z$, the latter column being nonzero, which is a contradiction.

Finally, if $r = q$, then $Z_{qs}$ is the nonzero block of the matrix $Z$. In this case the first row of the matrix $e_{1, N_{q-1}+i} \cdot Z$, where $N_{q-1} = n_1+n_2+ \cdots+ n_{q-1}$, is nonzero. Similarly, it leads to a contradiction, which completes the proof.   
\end{proof}
Note that conjugation of a maximal D$_q$ subalgebra of M$_n(K)$ is still a maximal D$_q$ subalgebra of M$_n(K)$, and so, by Theorem \ref{typical are maximal} and Theorem \ref{corollary block}, we have the following: 

\begin{corollary} \label{maximal Dq iff block Dq}
An algebra  $\mathcal{A}$  is a maximal D$_q$ subalgebra of M$_n(K)$ if and only if it is conjugated with 
a D$_q$ subalgebra of~{\rm M}$_n(K)$ of some type $(n_1, n_2, \ldots, n_q)$ with max-comm db's.
\end{corollary}

In the rest of this section we will examine conjugations, which satisfy some additional properties, of a D$_q$ subalgebra of M$_n(K)$ of some type with max-comm db's. We need the following result involving some matrix equations:

\begin{lemma} \label{lemma after discussion}
Let $r$, $s$ and $t$  be positive integers, and let $Y \in {\rm M}_{r \times t}(K)$, $W \in {\rm M}_{s \times t} (K)$, with $W \ne 0_{s \times t}$.
If  $YZW = 0_{r \times t}$  for all $Z \in {\rm M}_{t \times s}(K)$, then $Y=0_{r \times t}$. 
\end{lemma}
\begin{proof}
Write $Y = 
\left(
\begin{array}{ccc}
y_{11} & \ldots & y_{1t} \\
\vdots & \ddots & \vdots \\
y_{r1} & \ldots & y_{rt} 
\end{array}
\right), \  
W = 
\left(
\begin{array}{ccc}
w_{11} & \ldots & w_{1t} \\
\vdots & \ddots & \vdots \\
w_{s1} & \ldots & w_{st} 
\end{array}
\right),
$
with (say) $w_{ij}\ne 0$ (for some indices $i, j$, with $1 \leq i \leq s$, $1 \leq j \leq t$). Consider the matrix unit $e_{ki}\in{\rm M}_{t\times s}(K)$ for any fixed~$k, \ 1\le k\le t$. 

By assumption and direct calculation, we have 
\begin{equation}
0_{r \times t} =Y e_{ki} W =
\left(
\begin{array}{ccccc}
y_{1k}w_{i1} & \ldots & y_{1k} w_{ij} & \ldots & y_{1k} w_{it}  \\
y_{2k}w_{i1} & \ldots & y_{2k} w_{ij} & \ldots & y_{2k} w_{it}  \\
\vdots       & \ddots & \vdots        & \ddots & \vdots \\
y_{rk}w_{i1} & \ldots & y_{rk} w_{ij} & \ldots & y_{rk} w_{it} 
\end{array}
\right),
\end{equation}
and  so, since $w_{ij} \neq 0$, we conclude that 
$$y_{1k} =y_{2k} = \ldots = y_{rk} =0,$$ i.e., the $k$-th column of $Y$ is zero.
As $k$ was arbitrary, we conclude that $Y=0_{r \times t}.$
\end{proof}
Since ${\rm det}(X)={\rm det}(X_{11}) \cdot {\rm det} (X_{22})$ if $X$ is a block triangular matrix
$\left(
\begin{array}{cc}
X_{11}              & X_{12} \\
0_{n_2 \times n_1}  & X_{22}
\end{array}
\right),
$
where $X_{11}\in {\rm M}_{n_1}(K), \ X_{22}\in {\rm M}_{n_2}(K), \ X_{12} \in {\rm M}_{n_1 \times n_2}(K)$, with $n_1$ and $n_2$ positive integers, 
it is evident that $X_{11}$ and $X_{22}$ are invertible if $X$ is invertible, and direct matrix multiplication yields 

\begin{equation} \label{inverse of block triangular}
X^{-1} = 
\left(
\begin{array}{cc}
X_{11}^{-1}         & -X_{11}^{-1} X_{12} X_{22}^{-1} \\
0_{n_2 \times n_1}  & X_{22}^{-1}
\end{array}
\right).
\end{equation}


\begin{lemma} \label{block triangular conjugation}
Let $q,n_1, n_2, \ldots, n_q$  and $n$ be positive integers such that $n_1+n_2+\cdots+n_q=n$, and let $\mathcal{A}$ be a subalgebra 
of~M$_n(K)$. If 
$$
\left(
\begin{array}{ccccc}
0_{n_1} & {\rm M}_{n_1\times n_2}(K) & {\rm M}_{n_1\times n_3}(K) & \cdots      & {\rm M}_{n_1\times n_q}(K)\\ 
       \\
        & 0_{n_2}              & {\rm M}_{n_2\times n_3}(K) & \cdots      & {\rm M}_{n_2\times n_q}(K)  \\  \\
        \\
        &                      & \ddots               & \ddots      & \vdots                \\ 
        \\
        &                      &                      & \ddots      & {\rm M}_{n_{q-1}\times n_q}(K)\\  
        \\
        &                      &                      &             & 0_{n_q}           
\end{array}
\right)\subseteq \A
$$ 
and $X \in {\rm M}_n(K)$ is an invertible matrix such that $X^{-1} \mathcal{A} X$ is contained in
the full subalgebra of~M$_n(K)$ of type $(n_1, n_2, \ldots, n_q)$ (see Definition \ref{of type}), then matrix $X$ 
also belongs to the full subalgebra of~M$_n(K)$ of type $(n_1, n_2, \ldots, n_q)$.
\end{lemma}


\begin{proof}
The result that we want to prove is obvious if $q=1$. Thus, building a proof using mathematical induction, we start with $q =2 $, and positive integers $n_1, n_2$ and $n$ such that $n_1+n_2 = n$.
Write matrix $X$ in block form 
$
\left(
\begin{array}{cc}
X_{11}             & X_{12} \\
X_{21}             & X_{22}
\end{array}
\right),
$
where $X_{ij} \in {\rm M}_{n_i \times n_{j}}(K)$, $1 \leq i, j \leq 2$,  
and suppose, for the contrary, that block $X_{21} \neq 0_{n_2 \times n_1}$. Write also the inverse matrix $X^{-1}$ in block form 
$
\left(
\begin{array}{cc}
X'_{11}             & X'_{12} \\
X'_{21}             & X'_{22}
\end{array}
\right),
$ 
where $X'_{ij} \in {\rm M}_{n_i \times n_j}(K)$.
Then, by assumption, the matrix
$
\left(
\begin{array}{cc}
0_{n_1}        & Y \\
               & 0_{n_2}
\end{array}
\right)
$
belongs to $\mathcal{A}$ for every $Y \in {\rm M}_{n_1 \times n_2}(K)$. Therefore,
\begin{align*}
X^{-1} 
\left(
\begin{array}{cc}
0_{n_1}   & Y \\    
          & 0_{n_2}
\end{array}
\right) 
X &= 
\left(
\begin{array}{cc}
X'_{11}   & X'_{12} \\ 
X'_{21}   & X'_{22}
\end{array}
\right)
\left(
\begin{array}{cc}
0_{n_1}   & Y \\    
          & 0_{n_2}
\end{array}
\right) 
\left(
\begin{array}{cc}
X_{11}   & X_{12} \\ 
X_{21}   & X_{22}
\end{array}
\right)
\\
&=
\left(
\begin{array}{cc}
X'_{11} Y X_{21}    & X'_{11} Y X_{22}  \\ 
X'_{21} Y X_{21}    & X'_{21} Y X_{22}
\end{array}
\right)
\in \left(
\begin{array}{cc}
{\rm M}_{n_1}(K)          & {\rm M}_{n_1 \times n_2}(K) \\ 
0_{n_2 \times n_1}  & {\rm M}_{n_2}(K)
\end{array}
\right).
\end{align*}
Hence, $X_{21}'YX_{21} = 0_{n_2 \times n_1}$ for every $Y \in$ M$_{n_1 \times n_2}(K)$, and so Lemma \ref{lemma after discussion} implies that $X'_{21}= 0_{n_2 \times n_1}$. We conclude from formula (\ref{inverse of block triangular}) that $X_{21} = 0_{n_2 \times n_1}$. This is a contradiction, which completes the desired result for $q = 2$.

Assume now inductively that the result holds for some $q\ge2$, and let $n$ and $n_1, n_2, \ldots, n_{q+1}$ be positive integers  such that $n_1+n_2+ \cdots+ n_{q+1} = n$. 
Write matrix $X$ and its inverse in block form
$$
X  =
\left(
\begin{array}{ccc}
X_{11} & \ldots & X_{1,q+1} \\
\vdots & \ddots & \vdots \\
X_{q+1,1} & \ldots & X_{q+1,q+1} 
\end{array}
\right), 
\quad
X^{-1} =
\left(
\begin{array}{ccc}
X'_{11}    & \ldots    & X'_{1,q+1} \\
\vdots    & \ddots    & \vdots \\
X'_{q+1,1} & \ldots & X'_{q+1,q+1} 
\end{array}
\right),
$$
with $X_{ij}, X'_{ij} \in {\rm M}_{n_i \times n_j}(K)$ for all $1\le i,j\le q+1$. Firstly, we will show that $X_{j1} = 0_{n_j \times n_1}$ for $j=2,3,\ldots,q+1$. Then we will use the inductive assumption. Suppose, for the contrary, that $X_{j1} \neq 0_{n_j \times n_1}$ for some $j$, $2 \leq j \leq q+1$. Let $Y_{1j} \in {\rm M}_{n_1 \times n_j}(K)$ be an arbitrary matrix. Then 
$$
\left(
\begin{array}{cccccc}
 & \ldots & 0_{n_{1} \times n_{j-1}} & Y_{1j}               & 0_{n_1 \times n_{j+1}} & \ldots \\
 & \ldots & 0_{n_2 \times n_{j-1}}   & 0_{n_2 \times n_{j}} & 0_{n_2 \times n_{j+1}}  & \ldots \\
 &        & \vdots                   & \vdots               & \vdots             
\end{array}
\right) \in \mathcal{A},
$$
and direct calculation gives 
$$
X^{-1}
\left(
\begin{array}{cccccc}
 & \ldots & 0_{n_{1} \times n_{j-1}} & Y_{1j}               & 0_{n_1 \times n_{j+1}} & \ldots \\
 & \ldots & 0_{n_2 \times n_{j-1}}   & 0_{n_2 \times n_{j}} & 0_{n_2 \times n_{j+1}}  & \ldots \\
 &        & \vdots                   & \vdots               & \vdots             
\end{array}
\right)
X
=
\left(
\begin{array}{ccc}
X'_{11}Y_{1j}X_{j1} & X'_{11}Y_{j1}X_{j2} & \ldots \\
X'_{21}Y_{1j}X_{j1} & X'_{21}Y_{j1}X_{j2} & \ldots \\
\vdots              & \vdots              &  
\end{array}
\right).
$$
Since the algebra $X^{-1}\A X$ is contained in the full subalgebra of M$_n(K)$ of type $(n_1, n_2, \ldots, n_q)$,
we obtain the equalities 
$$
X'_{21}Y_{1j}X_{j1} = 0_{n_2 \times n_{1}}, \quad X'_{31} Y_{1j}X_{j1} = 0_{n_{3} \times n_1}, \quad \ldots\quad, X'_{q+1,1} Y_{1j}X_{j1} = 0_{n_{q+1} \times n_1}.
$$
We now apply Lemma \ref{lemma after discussion} to each of these equations and obtain that $X'_{21}, X'_{31}, \ldots, X'_{q+1,1}$ are zero matrices. Therefore we can write $X^{-1} =
\left(
\begin{array}{cc}
X'_{11}               & \overline{X'}_{12} \\
0_{N_2 \times n_1}   & \overline{X'}_{22}
\end{array}
\right), 
$
where $N_2 = n_2+n_3+\cdots+n_{q+1}$ and
$$
\overline{X'}_{12}=
\left(
\begin{array}{cccc}
X'_{12} & X'_{13} & \ldots & X'_{1,q+1}
\end{array}
\right),
\quad
\overline{X'}_{22}=
\left(
\begin{array}{ccc}
X'_{22} & \ldots & X'_{2,q+1} \\
\vdots & \ddots & \vdots \\
X'_{q+1,2} & \ldots & X'_{q+1,q+1} 
\end{array}
\right).
$$
From formula (\ref{inverse of block triangular}) on the inverse of a block triangular matrix it follows that  
$$X_{21} = 0_{n_2 \times n_1}, \quad X_{31} = 0_{n_3 \times n_1}, \quad \ldots, \quad   X_{q+1,1} = 0_{n_{q+1} \times n_1}.$$ 
This is a contradiction, since $X_{j1} \neq 0_{n_j \times n_1}$ for some $2 \leq j \leq q+1$. Hence, indeed $X_{21}$, $X_{31}$, \ldots, $X_{q+1,1}$ are zero matrices. Now we can use the inductive assumption to the subalgebra of ${\rm M}_{N_2}(K)$ obtained from the entries of $\mathcal{A}$ starting from row $n_1+1$ and column $n_1+1$. It implies that matrix~$X$ 
belongs to the full subalgebra of M$_n(K)$ of type $(n_1, n_2, \ldots, n_q)$.
\end{proof}

Although conjugation of 
a block triangular D$_q$ subalgebra $\mathcal{A}$ of M$_n(K)$ of some type $(n_1, n_2, \ldots, n_q)$
with an invertible matrix $X\in {\rm M}_n(K)$ can result in the subalgebra $X^{-1}\mathcal{A}X$ of M$_n(K)$ not being a block triangular subalgebra of M$_n(K)$, as shown in the example below, we will prove 
in~Proposition~\ref{conjugation block triangular} that this does not happen if $X$ is such $X^{-1}\mathcal{A}X$ is contained in the 
full subalgebra of M$_n(K)$ of type $(n_1, n_2, \ldots, n_q)$. 

\begin{example}
Let  $\mathcal{A}$ be a 
D$_4$ subalgebra of M$_n(K)$ of type $(n_1, n_2, n_3, n_4)$, possibly with max-comm db's, and consider the invertible block matrix 
\[
 X=\left(
 \begin{array}{cccc}
 0_{n_1 \times n_4} & 0_{n_1 \times n_2} & 0_{n_1 \times n_3} & I_{n_1} \\
 0_{n_2 \times n_4} & I_{n_2}            & 0_{n_2 \times n_3} & 0_{n_2 \times n_1} \\
 0_{n_3 \times n_4} & 0_{n_3 \times n_2} & I_{n_3}            & 0_{n_3 \times n_1} \\
I_{n_4}             & 0_{n_4 \times n_2} & 0_{n_4 \times n_3} & 0_{n_4 \times n_1}
\end{array}
\right)\in {\rm M}_n(K),
\]
\noindent which clearly is a "block" version of the permutation matrix
\[
\left(
 \begin{array}{cccc}
 0 & 0 & 0 & 1 \\
 0 & 1 & 0 & 0 \\
 0 & 0 & 1 & 0 \\
 1 & 0 & 0 & 0
\end{array}
\right).
\]
\noindent Then
\[
X^{-1} = 
 \left(
 \begin{array}{cccc}
 0_{n_4 \times n_1} & 0_{n_4 \times n_2} & 0_{n_4 \times n_3} & I_{n_4} \\
 0_{n_2 \times n_1} & I_{n_2}            & 0_{n_2 \times n_3} & 0_{n_2 \times n_4} \\
 0_{n_3 \times n_1} & 0_{n_3 \times n_2} & I_{n_3}            & 0_{n_3 \times n_4} \\
I_{n_1}             & 0_{n_1 \times n_2} & 0_{n_1 \times n_3} & 0_{n_1 \times n_4}
\end{array}
\right).
\]
\noindent Writing the algebra $\mathcal{A}$ in block triangular form

 \[
 \left(
 \begin{array}{cccc}
 \mathcal{A}_{11} & \mathcal{A}_{12} & \mathcal{A}_{13} & \mathcal{A}_{14} \\
                  & \mathcal{A}_{22} & \mathcal{A}_{23} & \mathcal{A}_{24} \\
                  &                  & \mathcal{A}_{33} & \mathcal{A}_{34} \\
                  &                  &                  & \mathcal{A}_{44}
 \end{array}
 \right),
\]
as in (\ref{block triang}), direct verification yields
\[
 X^{-1} \mathcal{A} X =
 \left(
 \begin{array}{cccc}
 \mathcal{A}_{44} & 0_{n_4 \times n_2} & 0_{n_4 \times n_3} & 0_{n_4 \times n_1} \\
 \mathcal{A}_{24} & \mathcal{A}_{22}   & \mathcal{A}_{23}   & 0_{n_2 \times n_1} \\
 \mathcal{A}_{34} & 0_{n_3 \times n_2} & \mathcal{A}_{33}   & 0_{n_3 \times n_1} \\
 \mathcal{A}_{14} & \mathcal{A}_{12}   & \mathcal{A}_{13}   & \mathcal{A}_{11}
 \end{array}
 \right),
\]
implying that, by Definition \ref{of type}, $X^{-1} \mathcal{A} X$ is not 
a subalgebra of M$_n(K)$ 
of any type $(\ell_1, \ell_2, \ell_3, \ell_4)$. 
\end{example}

\begin{proposition} \label{conjugation block triangular}
Let $\mathcal{A}$ be a 
{\rm D}$_q$ subalgebra of~{\rm M}$_n(K)$ of type $(n_1, n_2, \ldots, n_q)$ with max-comm db's, 
and let $X \in {\rm M}_n(K)$ be an invertible matrix such that $X^{-1}\mathcal{A}X$ is contained in the 
full subalgebra of M$_n(K)$ of type~$(n_1, n_2, \ldots, n_q)$.
Then $X^{-1}\mathcal{A}X$ 
is also a
{\rm D}$_q$ subalgebra of~{\rm M}$_n(K)$ of type~$(n_1, n_2, \ldots, n_q)$ with max-comm db's.
Moreover, if we write $\A_{ii}$ and $\A'_{ii}$, $i = 1, 2, \ldots, q,$ for the diagonal blocks of the subalgebras $\A$ and $X^{-1} \A X$ (see \ref{block triang}), respectively, then $\A_{ii}$ and 
 $\A'_{ii}$ are conjugates.
\end{proposition}

\begin{proof}
By Lemma \ref{block triangular conjugation}, $X$ is in the full subalgebra of M$_n(K)$ of type~$(n_1, n_2, \ldots, n_q)$, and so 
$$
X=\left(
\begin{array}{ccc}
X_{11} & \ldots & X_{1q} \\
       & \ddots & \vdots \\
       &        & X_{qq}
\end{array}
\right),
$$
for some $X_{ij} \in {\rm M}_{n_i \times n_j}(K)$, $1 \leq i \leq j \leq q$. Write the inverse $X^{-1}$ in block form
$$
\left(
\begin{array}{ccc}
X'_{11}  & \ldots & X'_{1q} \\
\vdots   & \ddots & \vdots \\
X'_{q1}  & \ldots & X'_{qq}
\end{array}
\right),
$$
for $X_{ij}' \in {\rm M}_{n_i \times n_j}(K), \ 1\le i,j\le q$. We will show that $X^{-1}$ is also block triangular. 

Divide $X$ into four blocks of sizes $n_1 \times n_1$, $n_1 \times (n_2+ \cdots+n_q)$, $(n_2+\cdots+n_q) \times n_1$ and $(n_2+ \cdots+n_q) \times (n_2+ \cdots+n_q)$, respectively.  Formula (\ref{inverse of block triangular}) implies that $X'_{11} = X_{11}^{-1}$ and that the matrices $X'_{21}, X'_{31}, \ldots, X'_{q1}$ are all zero matrices. These facts lead us to finding the inverse of the block triangular matrix
$$
\left(
\begin{array}{ccc}
X_{22} & \ldots & X_{2q} \\
       & \ddots & \vdots \\
       &        & X_{qq}
\end{array}
\right).
$$
Continuing in the above way, we finally get that $X^{-1}$ is block triangular with $X'_{ii} =X_{ii}^{-1}$ for every~$i, \ i=1,2,\ldots,q.$

Now write 
the 
D$_q$ subalgebra $\A$ of M$_n(K)$ with max-comm db's 
in the form 
$$
\left(
\begin{array}{ccc}
\mathcal{A}_{11} & \ldots & \mathcal{A}_{1q} \\
                 & \ddots & \vdots \\
                 &        & \mathcal{A}_{qq}
\end{array}
\right),    
$$
as in (\ref{block triang}). 
Since the matrices $X$ and $X^{-1}$ are block triangular, it follows that $X^{-1} \A X$ is contained in the  
D$_q$~subalgebra 
\begin{equation} \label{block conjugated}
\left(
\begin{array}{ccccc}
X_{11}^{-1} \A_{11}X_{11} & {\rm M}_{n_1\times n_2}(K)   & {\rm M}_{n_1\times n_3}(K) & \cdots    & {\rm M}_{n_1\times n_q}(K)\\ 
        \\
                           & X_{22}^{-1} \A_{22}X_{22}   & {\rm M}_{n_2\times n_3}(K) & \cdots    & {\rm M}_{n_2\times n_q}(K)  \\  \\
           \\
                            &                            & \ddots               & \ddots      & \vdots                \\ 
           \\
                            &                            &                      & \ddots      & {\rm M}_{n_{q-1}\times n_q}(K)\\  
           \\
                            &                            &                      &             & X_{qq}^{-1} \A_{qq} X_{qq}           
\end{array}
\right)
\end{equation}
of M$_n(K)$ of type $(n_1, n_2, \ldots, n_q)$ with max-comm db's.
We will prove that the reverse inclusion also holds.

We first show that
\begin{equation} \label{1q block}
\left(
\begin{array}{ccc}
\ldots               & 0_{n_{n_1 \times q-1}}  & {\rm M}_{n_1 \times n_q}(K) \\
\ldots               & 0_{n_2 \times n_{q-1}}  & 0_{n_2 \times n_q}         \\
                     & \vdots                  & \vdots 
\end{array}
\right)
\subseteq X^{-1} \A X.
\end{equation}
To this end, keep in mind that every $X_{ii}, \ i=1,2,\ldots,q,$ is invertible, and note that for an arbitrary matrix $Y_{1q} \in {\rm M}_{n_1 \times n_q}(K)$,
$$
\left(
\begin{array}{ccc}
\ldots & 0_{n_1 \times n_{q-1}} & Y_{1q} \\
\ldots & 0_{n_2 \times n_{q-1}} & 0_{n_2 \times n_q}  \\
       & \vdots                 & \vdots 
\end{array}
\right) 
=
X^{-1}
\left(
\begin{array}{ccc}
\ldots & 0_{n_1 \times n_{q-1}} & X_{11}Y_{1q}X^{-1}_{qq} \\
\ldots & 0_{n_2 \times n_{q-1}} & 0_{n_2 \times n_q}  \\
       & \vdots                 & \vdots 
\end{array}
\right)
X
 \in X^{-1} \A X,
$$
which establishes (\ref{1q block}).
Next take an arbitrary matrix $Y_{1,q-1} \in {\rm M}_{n_1 \times n_{q-1}}(K)$. Then 
\begin{align*}
X^{-1}
&\left(
\begin{array}{cccc}
\ldots & 0_{n_1 \times n_{q-2}}  & X_{11}Y_{1,q-1}X^{-1}_{q-1,q-1} & 0_{n_1 \times n_q} \\
\ldots & 0_{n_2 \times n_{q-2}}  & 0_{n_2 \times n_{q-1}} & 0_{n_2 \times n_q}  \\
       & \vdots                 & \vdots                  &  \vdots
\end{array}
\right)
X = 
\\
=&
\left(
\begin{array}{cccc}
\ldots & 0_{n_1 \times n_{q-2}}  & Y_{1,q-1} & Z_{1,q} \\
\ldots & 0_{n_2 \times n_{q-2}}  & 0_{n_2 \times n_{q-1}} & 0_{n_2 \times n_q}  \\
       & \vdots                 & \vdots                  & \vdots
\end{array}
\right) \in X^{-1} \A X
\end{align*}
for some $Z_{1,q-1}\in {\rm M}_{n_1\times n_q}(K)$.
Because of the inclusion in $(\ref{1q block})$, we deduce that 
$$
\left(
\begin{array}{cccc}
\ldots & 0_{n_1 \times n_{q-2}}  & Y_{1,q-1} & 0_{n_1 \times n_q} \\
\ldots & 0_{n_2 \times n_{q-2}}  & 0_{n_2 \times n_{q-1}}          & 0_{n_2 \times n_q}  \\
       & \vdots                  & \vdots                          & \vdots
\end{array}
\right)\in X^{-1} \A X.
$$
Therefore, $X^{-1} \A X$ contains 
$$
\left(
\begin{array}{cccc}
\ldots               & 0_{n_{n_1 \times q-2}}  & {\rm M}_{n_1 \times n_{q-1}}(K) & 0_{n_1 \times n_{q}} \\
\ldots               & 0_{n_2 \times n_{q-2}}  & 0_{n_2 \times n_{q-1}}          & 0_{n_2 \times n_{q}} \\
                     & \vdots                  & \vdots                          & \vdots 
\end{array}
\right),
$$
Continuing in this way we obtain the inclusion
$$
\left(
\begin{array}{ccccc}
X_{11}^{-1} \A_{11} X_{11} & {\rm M}_{n_1 \times n_2}(K) & {\rm M}_{n_1 \times n_3}(K) & \ldots & {\rm M}_{n_1 \times n_q}(K) \\
                 &  0_{n_2}                    & 0_{n_2 \times n_3}        & \ldots &  0_{n_1 \times n_q}    \\
                 &                             & 0_{n_3}                   & \ldots &  0_{n_3 \times n_q}    \\
                 &                             &                           & \ddots &  \vdots                \\
                 &                             &                           &        &  0_{n_q}               \\
\end{array}                
\right) \subseteq X^{-1} \A X
$$
Proceeding in the same manner we obtain that
$$
\left(
\begin{array}{ccccc}
0_{n_1}          &  0_{n_1 \times n_2}         & 0_{n_1 \times n_3}          & \ldots & 0_{n_1 \times n_q} \\
                 &  X_{22}^{-1} \A_{22} X_{22} & {\rm M}_{n_2 \times n_3}(K) & \ldots & {\rm M}_{n_2 \times n_q}(K)    \\
                 &                             & 0_{n_3}                     & \ldots &  0_{n_3 \times n_q}    \\
                 &                             &                             & \ddots &  \vdots                \\
                 &                             &                             &        &  0_{n_q}               \\
\end{array}                
\right) \subseteq X^{-1} \A X.
$$
The pattern is now clear, and so finally we conclude that
$X^{-1}\A X$ contains the entire subalgebra of~M$_n(K)$ in (\ref{block conjugated}), which completes the proof.
\end{proof}

\begin{corollary}\label{upper triangular max}
Let $\mathcal{A}$ be a subalgebra of {\rm U}$_n(K)$.
\begin{enumerate}
\item 
If $\mathcal{A}$ is a maximal D$_q$ subalgebra of {\rm M}$_n(K)$, then $\mathcal{A}$ is a D$_q$ subalgebra of {\rm M}$_n(K)$ of some 
type~$(n_1,n_2,\ldots, n_q)$ with max-comm db's.

\item If $\mathcal{A}$ is a D$_q$ subalgebra of 
{\rm M}$_n(K)$ with maximum dimension, then $\mathcal{A}$ is a max-dim D$_q$ subalgebra of {\rm M}$_n(K)$ of some 
type~$(n_1,n_2,\ldots, n_q)$ with max-dim db's.

\end{enumerate}
\end{corollary}

\begin{proof}
(1) By Theorem \ref{corollary block}, there is an invertible matrix $X$ such that $X^{-1}\A X$ is a D$_q$ subalgebra of M$_n(K)$ of some 
type~$(n_1,n_2,\ldots, n_q)$ with max-comm db's. The desired result now follows immediately from Proposition \ref{conjugation block triangular}, since $\A= X(X^{-1}\A X)X^{-1}$. 

(2) By (1), $\mathcal{A}$ is a D$_q$ subalgebra of M$_n(K)$ of some 
type~$(n_1,n_2,\ldots, n_q)$ with max-comm db's. 
Had any of the commutative algebras in the diagonal blocks of $\A$ not been of maximum dimension, we would have been able to replace it by a commutative algebra with larger dimension, thereby obtaining a D$_q$ subalgebra of M$_n(K)$ with dimension larger than that of $\A$; a contradiction.
\end{proof}

The construction of a D$_q$ subalgebra of M$_n(K)$ with maximum dimension described in Theorem~\ref{precise},
 combined with the examples in (\ref{typical commutative 1}) and in (\ref{typical commutative 2}) of commutative subalgebras of M$_n(K)$
 with maximum dimension, gives an example of a D$_q$ subalgebra of M$_n(K)$ with maximum dimension contained in U$_n(K)$. So if $\A$ is a D$_q$ subalgebra of U$_n(K)$ with maximum dimension, then $\A$ is a D$_q$ subalgebra of M$_n(K)$ with maximum dimension.  
Consequently, Corollary~\ref{upper triangular max}(2) confirms the ``underlying conjecture" embodied in Question \ref{question typical} by answering the question in the negative. 

We conclude the section with a characterization of when D$_q$ subalgebras 
of~M$_n(K)$ with max-comm db's are conjugated.

\begin{theorem} \label{max up to conj}
Let $\A$ and $\B$ be {\rm D}$_q$ subalgebras of~{\rm M}$_n(K)$ of types $(n_1, n_2, \ldots, n_q)$ and $(\ell_1, \ell_2, \ldots, \ell_q)$, respectively, with max-comm db's. Write $\A_{ii}$ and $\B_{ii}$, $i = 1, 2, \ldots, q,$ for the diagonal blocks of the subalgebras $\A$ and $\B$, respectively (see \ref{block triang}). Then $\A$ and $\B$ are conjugates if and only if the $q$-tuples $(n_1, n_2, \ldots, n_q)$ and $(\ell_1, \ell_2, \ldots, \ell_q)$ are equal and, for every $i, \ i =1, 2, \ldots, q, \ \A_{ii}$ and 
 $\B_{ii}$ are conjugates.
\end{theorem}
\begin{proof}
Firstly, assume that $\A$ and $\B$ are {\rm D}$_q$ subalgebras of~{\rm M}$_n(K)$ of the same type $(n_1, n_2, \ldots, n_q)$ with max-comm db's such that for every $j, \ j=1, 2, \ldots, q$, the diagonal blocks $\A_{jj}$ and $\B_{jj}$  are conjugates. Then there are invertible matrices $X_{jj} \in {\rm M}_{n_j}(K)$ such that $X_{jj}^{-1}\A_{jj}X_{jj} = \B_{jj}$. 
We will show that
\begin{equation} \label{isomorphis and conjugation}
\left(
\begin{array}{cccc}
X^{-1}_{11} &        &        & \\
       & X^{-1}_{22} &        & \\
       &        & \ddots & \\
       &        &        & X^{-1}_{qq}
\end{array}
\right)
\A
\left(
\begin{array}{cccc}
X_{11} &        &        & \\
       & X_{22} &        & \\
       &        & \ddots & \\
       &        &        & X_{qq}
\end{array}
\right)
=
\B.
\end{equation}
Note that, by Proposition \ref{conjugation block triangular}, the subalgebra on the left hand side in (\ref{isomorphis and conjugation}) is a 
{\rm D}$_q$~subalgebra of~{\rm M}$_n(K)$ of type $(n_1, n_2, \ldots, n_q)$ with max-comm db's. It is easy to check that this subalgebra has $X_{jj}^{-1} \A_{jj} X_{jj}$ as its diagonal blocks, which establishes the equality in (\ref{isomorphis and conjugation}).

Conversely, assume that the algebras $\A$ and $\B$ are conjugated. If we can show that tuples the  $q$-tuples $(n_1, n_2, \ldots, n_q)$ and $(\ell_1, \ell_2, \ldots, \ell_q)$ are equal, then the result follows directly from 
Proposition~\ref{conjugation block triangular}. To this end, first note that, by Proposition \ref{proposition power of ideal generated by commutators}, 
$$n_1+n_2+ \cdots +n_i = {\rm dim}_K \C_\A^{q-i}V$$
for $i =1, 2, \ldots, q-1$, where $V = K^n$ and $\C_\A$ is the ideal of $\A$ generated by all $[x,y], \ x,y \in \A$. 
Since $\sum_{j=1}^q n_j =n$, we can write
$$
n_i =
\left\{
\begin{array}{ll}
{\rm dim}_K \C_\A^{q-i}V-{\rm dim}_K \C_\A^{q-i+1}V,& \text{ if } i =1, 2, \ldots, q-1 \\
n-{\rm dim}_K \C_\A V                                 ,& \text{ if } i = q. 
\end{array}
\right.
$$
Hence the type of the subalgebra $\A$ of  M$_n(K)$ is determined by $\dim_K \mathcal{\C}_{\A}^{i}V$ for $i=1, 2, \ldots, q-1$. Similarly, the type of the subalgebra $\B$ is determined by $\dim_K \C_{\B}^i V$, where $\C_\B$ is the ideal of $\B$ generated by all $[x,y], \ x,y \in \B$. 

By assumption, there exists an invertible matrix $X \in {\rm M}_n(K)$ such that $X^{-1}
\A X = \B$,
which implies that $\C_\B^i = (X^{-1} \C_\A X)^i = X^{-1} \C_\A^i X $ for all $i = 1, 2 \ldots, q-1$. To complete the proof we will show that ${\rm dim}_K \C_\A^i V= {\rm dim}_K \C_\B^i V$.

If vectors $C_1 v_1, C_2 v_2, \ldots, C_k v_k$ constitute
a basis of $\C_\A^i V$ for some matrices $C_j \in \C_\A^i$ and some vectors $v_j \in V$, then it can be shown directly that the vectors $X^{-1} C_j v_j = (X^{-1} C_j X) X^{-1}v_j, \ j =1, 2, \ldots, k,$ of the vector space $(X^{-1} \C_\A^i X) V$ are linearly independent. Therefore, ${\rm dim}_K \C_\A^i V \leq \dim_K (X^{-1} \C_\A^i X) V = \dim_K \C_\B^iV $. Similarly, we can show that  $\dim_K \C_\B^iV = \dim_K(X^{-1} \C_\A^i X) V \leq {\rm dim}_K \C_\A^i V$ by taking basis vectors of the vector space $(X^{-1} \C_\A^i X) V$ and producing linearly independent vectors in the vector space $\C_\A^iV$. 
\end{proof}

\section{Remarks on a result by Jacobson} \label{remarks on jacobson}

 In this section, we clarify the structure of commutative subalgebras  of M$_n(K)$, with $K$ 
 an algebraically closed field, 
 as discussed in \cite{Jac}.
 
 Throughout this section $K$ is an algebraically closed field, and $\A$ is a commutative subalgebra  of~M$_n(K)$ with maximum dimension. We focus in particular on the structure of $\A$ for $n\in\{2,3\}$, which, in the light of the footnote in \cite[page 436]{Jac}, does not seem to be so readily obtained after all. These considerations will be used in the subsequent sections.

By \cite{Jac}, for  $n >3$ we have the following:
\begin{itemize}
\item If $n$ is an even integer ($n = 2\ell$), then $\A$ is conjugated to 
\begin{equation}\label{commJacob0}
C^1_{2\ell}(K) := 
KI_{n}+\left(
\begin{array}{cc}

0_{\ell} & {\rm M}_\ell(K) \\
         &  0_{\ell}

\end{array}
\right).
\end{equation}
\item If $n$ is an odd integer ($n = 2 \ell +1$), then $\A$ is conjugated to 
\begin{equation}\label{commJacob1}
C^1_{2\ell+1}(K) :=
KI_{n}+\left(
\begin{array}{cc}

0_{\ell} & {\rm M}_{\ell \times (\ell+1)}(K) \\
         &  0_{\ell+1}

\end{array}
\right)
\end{equation}
or
\begin{equation}\label{commJacob2}
C^2_{2\ell+1}(K) := 
KI_{n}+\left(
\begin{array}{cc}

0_{\ell+1} & {\rm M}_{(\ell+1) \times \ell}(K) \\
         &  0_{\ell}

\end{array}
\right).
\end{equation}
\end{itemize}
It is possible to show (see Corollary \ref{isomnotconj}) that the algebras $C^1_{2\ell+1}(K)$ and $C^2_{2\ell+1}(K)$ are not conjugated. However, they are isomorphic. Indeed, it is easily verified that the map from $C^1_{2\ell+1}(K)$ to~$C^2_{2\ell+1}(K)$ which rotates the rectangular block ${\rm M}_{\ell \times (\ell+1)}(K)$ counterclockwise through $90^\circ$ is an isomorphism. 

Next, let $n\le3.$ Obviously, for $n=1$ we have 
\begin{equation}\label{fornequal1}
\A = K =: C^1_1(K).
\end{equation}

We now carefully study the two remaining cases, i.e., when  $n \in \{2, 3\}$. Either  $\A$ is isomorphic to a nontrivial product  $\A_1 \times \A_2$ of algebras, or it is not isomorphic to such a product. Suppose that we have the latter situation.
Since  $K$ is algebraically closed and $\A$ is a finite dimensional commutative algebra, it follows from
\cite[(3.5) Wedderburn-Artin Theorem]{Lam} that $$\A/J(\A) \cong K_1 \times K_2 \times \ldots \times K_t$$ 
for some $t\ge1$, with $K_i = K$ for every $i$. Since $J(\A)$ is nilpotent (see, for example, \cite[(4.12)~Theorem]{Lam}), it follows 
 from \cite[(21.28)~Theorem]{Lam} that if $t> 1$, then, lifting the idempotent $(1,0, \ldots, 0)$ of~$\A/J(\A)$, we get a nontrivial (i.e., $e \notin \{0, 1\}$) idempotent  $e \in \A$.
 Therefore, $$\A \cong e\A e \times (1-e)\A(1-e),$$ which contradicts our assumption. Consequently, 
$\A$ is local, and so by \cite[Proposition 10]{JJLM}, $\A$ is conjugated to some subalgebra $C$ of U$^*_n(K)$. For $C$ we have the following candidates:
\begin{itemize}
\item For  $n=2$ and the (commutative) algebra U$^*_2(K)$ we have only one possibility, namely
\begin{equation}\label{commJacob3}
C^1_{2}(K) = {\rm U}^*_2(K) =
KI_{2}+\left(
\begin{array}{cc}
0 & K \\
         &  0
\end{array}
\right).
\end{equation}

\item For  $n = 3$ and the algebra U$^*_3(K)$, taking an arbitrary $x \in K$, the matrix
$\left(
\begin{array}{ccc}
0 & 0 & x \\
  & 0 & 0 \\
  &   & 0
\end{array}
\right)
$
commutes with every $X\in {\rm U}^*_3(K)$. So we can see that $\left(
\begin{array}{ccc}
0 & 0 & x \\
  & 0 & 0 \\
  &   & 0
\end{array}
\right)\in C$, and consequently, 
$\left(
\begin{array}{ccc}
0 & 0 & K \\
  & 0 & 0 \\
  &   & 0
\end{array}
\right) \subseteq C
$. Since the maximal possible dimension for a commutative subalgebra of~${\rm M}_3(K)$ is  $1+\left \lfloor \frac{3^2}{4} \right  \rfloor = 3$, we deduce that there exist  $\alpha, \beta \in K$, not both equal to zero, such that $C=\left \{ \left(
\begin{array}{ccc}
a & \alpha b  & c \\
  & a        & \beta b  \\
  &          & a
\end{array}
\right) : a, b, c \in K
\right \}.
$

For   $\alpha \neq 0$ and $\beta = 0$ we get 
\begin{equation}\label{commJacob4}
C^1_{3}(K) = 
 KI_3 + \left(
\begin{array}{ccc}
0 & K        & K \\
  & 0        & 0\\
  &          & 0
\end{array}
\right).
\end{equation}

For $\alpha = 0$ and $\beta \neq 0$ we get 
\begin{equation}\label{commJacob5}
C^2_{3}(K) = 
 KI_3+ \left(
\begin{array}{ccc}
0 & 0        & K \\
  & 0        & K\\
  &          & 0
\end{array}
\right).
\end{equation}

Finally, for $\alpha \neq 0$ and $\beta \neq 0$, conjugation by 
$
 \left(
\begin{array}{ccc}
\alpha       & 0        & 0 \\
             & 1        & 0 \\
             &          & \beta^{-1}
\end{array}
\right)
$ 
gives
$$
 \left(
\begin{array}{ccc}
\alpha^{-1}  & 0        & 0 \\
             & 1        & 0 \\
             &          & \beta
\end{array}
\right)
 \left(
\begin{array}{ccc}
a & \alpha b        & c \\
  & a        & \beta b \\
  &          & a
\end{array}
\right)
 \left(
\begin{array}{ccc}
\alpha & 0        & 0 \\
  & 1        & 0\\
  &          & \beta^{-1}
\end{array}
\right)=
 \left(
\begin{array}{ccc}
a & b        & \alpha^{-1} \beta^{-1} c \\
  & a        & b \\
  &          & a
\end{array}
\right),$$ 
and so
\begin{equation}\label{commJacob6}
C^3_{3}(K) = 
KI_3 +  K \left(
\begin{array}{ccc}
0 & 1        & 0 \\
  & 0        & 1\\
  &       & 0
  \end{array}
\right) +
\left(
\begin{array}{ccc}
0 & 0        & K \\
  & 0        & 0\\
  &       & 0
  \end{array}
\right).
\end{equation}

\end{itemize}

\smallskip

Now  we consider the case where $\A$ is isomorphic to a proper product of algebras  $\A\cong \A_1 \times \A_2$. 
In this situation,  $1=e_1+ e_2$ for some nontrivial idempotents $e_1, e_2$.  Let  $\epsilon_1 \colon K^n \to K^n$ be the linear map such that $e_1$ is its matrix in the standard basis. Then as $\epsilon_1^2 = \epsilon_1$ we get
$K^n = {\rm im}(\epsilon_1) \oplus \ker(\epsilon_1)$. If  $v_1, v_2, \ldots, v_k$ form a basis for ${\rm im}(\epsilon_1)$, and  $v_{k+1}, v_{k+2}, \ldots, v_n$ form a basis for  $\ker(\epsilon_1)$ then in the basis $B = (v_1, v_2, \ldots, v_n)$ for $K^n$ we have
$ M(\epsilon_1)_{B}^{B} = 
\left(
\begin{array}{c|c}
I_k & 0_{k \times (n-k)} \\ \hline
    & 0_{n-k}
\end{array}
\right).
$
In the same basis $B$, for the map $\epsilon_2$ related to $e_2 = 1-e_1$, we have 
$M(\epsilon_2)_{B}^{B} = I_n - M(\epsilon_1)_{B}^{B} =
\left(
\begin{array}{c|c}
0_k & 0_{k \times (n-k)} \\ \hline
    & I_{n-k}
\end{array}
\right).
$

Since the idempotents  $e_1$ and $ e_2$ are orthogonal and  the algebra $\A$ is commutative, we have
$a = 1 \cdot a \cdot 1 = (e_1+e_2) a (e_1+e_2) = e_1 a e_1 + e_2 a e_2$ for every  $a \in \A$ . Therefore, as a vector space,  $\A = e_1 \A e_1 \oplus e_2 \A e_2$. Considering conjugation of $\A$ with the change-of-basis matrix from the standard basis for $K^n$ to $B$, we get
\begin{align*}
\A' =& \left(
\begin{array}{c|c}
I_k & 0_{k \times (n-k)} \\ \hline
    & 0_{n-k}
\end{array}
\right)\A'
\left(
\begin{array}{c|c}
I_k & 0_{k \times (n-k)} \\ \hline
    & 0_{n-k}
\end{array}
\right)
\oplus
\left(
\begin{array}{c|c}
0_k & 0_{k \times (n-k)} \\ \hline
    & I_{n-k}
\end{array}
\right)
\A'
\left(
\begin{array}{c|c}
0_k & 0_{k \times (n-k)} \\ \hline
    & I_{n-k}
\end{array}
\right)=\\
=&
\left(
\begin{array}{c|c}
\A_1' & 0_{k \times (n-k)} \\ \hline
     & \A_2'
\end{array}
\right),
\end{align*}
where $\A'$ is the conjugation of $\A$, and $\A'_1 \subseteq {\rm M}_{k}(K)$ and $\A'_2 \subseteq {\rm M}_{n-k}(K)$ are commutative subalgebras of M$_n(K)$ with maximum dimensions isomorphic to  $\A_1$ and  $\A_2$, respectively.
\begin{itemize}
\item
For  $n = 2 $ we have only one possibility, namely,  $\A$ is conjugated with 
\begin{equation}\label{commJacob111}
C^2_{2}(K)=\left(
\begin{array}{cc}
K & 0 \\
  & K
\end{array}
\right).
\end{equation}

\item 
For $n = 3$, either both algebras  $\A_1$ and $\A_2$ are  indecomposable or exactly one is indecomposable (recall that the maximal possible dimension of a commutative subalgebra of  ${\rm M}_3(K)$ is 3). In the first case, at first glance, $\A$ is conjugated   either with
\begin{equation*}
\Lambda_1 = \Biggl\{\left(
\begin{array}{ccc}
a & b        & 0 \\
  & a        & 0\\
  &          & c
  \end{array}
\right): a, b, c\in K\Biggr\}
\text{\,\, or\,\, } 
\Lambda_2 = \Biggl\{\left(
\begin{array}{ccc}
a & 0        & 0 \\
  & b        & c\\
  &          & b
  \end{array}
\right): a, b, c\in K\Biggr\}.
\end{equation*}
Note, however,  that considering  the invertible matrix 
$
Z = \left(
\begin{array}{ccc}
0 & 1        & 0 \\
0 & 0        & 1 \\
1 & 0        & 0
\end{array}
\right)
$ we have
\begin{equation*}
Z^{-1} \Lambda_1Z= \Lambda_2.
\end{equation*}
Thus $\A$ in this case is conjugated with 
\begin{equation}\label{commJacob7} 
C^4_{3}(K) = \Lambda_1 = \Biggl\{\left(
\begin{array}{ccc}
a & b        & 0 \\
  & a        & 0\\
  &          & c
  \end{array}
\right): a, b, c\in K\Biggr\}.
\end{equation}

\smallskip

In the case where exactly one of $\A_1$ and $\A_2$ is indecomposable, we get that $\A$ is conjugated with 
\begin{equation}\label{lastone}
C^5_{3}(K) = 
\left(
\begin{array}{ccc}
K & 0        & 0 \\
  & K        & 0 \\
  &          & K
\end{array}
\right).
\end{equation}
\end{itemize} 

Amongst the algebras  $C^1_1(K), C^1_2(K), C^2_2(K), C^1_3(K), C^2_3(K), C^3_3(K), C^4_3(K)$ and $C^5_3(K)$,  only $C^1_3(K)$ and $C^2_3(K)$ are isomorphic. However, by Corollary \ref{isomnotconj} these isomorphic subalgebras are not conjugated.

We summarize the above considerations (in this section) as follows:

\medskip

\begin{remark} \label{typical commutative}
Every commutative subalgebra of M$_n(K)$ with maximum dimension is conjugated with precisely one of these presented in $(\ref{commJacob0})-(\ref{lastone})$. Amongst these algebras,  $C^1_{2\ell+1}(K)$ and $C^2_{2\ell+1}(K)$ are isomorphic for every $\ell \geq 1$.
\end{remark}

By \cite{Jac} and the presentation above for algebraically closed fields $K$, we have determined, up to conjugation, all commutative subalgebras of M$_n(K)$ with maximum dimension. In fact, in  \cite{Jac} for~$n>3$, there is even  a weaker assumption on the field $K$ (not imperfect of characteristic $2$), but in the above characterization of commutative subalgebras of M$_2(K)$ and M$_3(K)$ with maximum dimension, the assumption that $K$ is algebraically closed is used. For example, in the case  of the field~$\mathbb{R}$ of real numbers,
$
\left\{
\left(
\begin{array}{cc}
a & -b \\
b & a 
\end{array}
\right) : 
a, b \in \mathbb{R}
\right\}
$ is a 2-dimensional subalgebra of M$_2(\mathbb{R})$ (isomorphic to the field of complex numbers) which is not conjugated with $C_2^{1}(K)$ nor $C_2^{2}(K)$. 

We conclude the section by showing that, for some of the commutative subalgebras $\A$ of M$_n(K)$ with maximum dimension, namely those $\A$'s presented in $(\ref{commJacob0}) - (\ref{commJacob5})$, the only possible conjugation of $\A$ which is contained in U$_n(K)$ is equal to $\A$ itself.
We will use this result to give an exact description of D$_q$ subalgebras of U$_n(K)$ with maximum dimension in Corollary \ref{max dim alg closed}.   

\medskip

\begin{proposition} \label{conjugation of Cn}
Let $\A$ be a commutative subalgebra of M$_n(K)$ with maximum dimension 
presented in one of $(\ref{commJacob0}) - (\ref{commJacob5})$.  If $X \in {\rm M}_n(K)$ is invertible and $X^{-1}\A X \subseteq {\rm U}_n(K)$, then $X^{-1} \mathcal{A} X = \A$.  
\end{proposition}

\begin{proof}
If $n=1$, then $\A = K$ and
there is nothing to prove. Assume now that $n\ge2$, with
$\A$ equal to one of the algebras in $(\ref{commJacob0})-(\ref{commJacob2})$ or $(\ref{commJacob3})-(\ref{commJacob5})$.
Then there exist positive integers $r$ and $s$ such that $r+s = n$ and 
$$
\A = KI_n + 
\left(
\begin{array}{cc}
0_{r}               & {\rm M}_{r \times s}(K) \\
0_{s \times r}      & 0_{s}
\end{array}
\right),
$$
and so we have the inclusions
$$
\left(
\begin{array}{cc}
0_{r}               & {\rm M}_{r \times s}(K) \\
0_{s \times r}      & 0_{s}
\end{array}
\right)
\subseteq 
\A
\quad {\rm and} \quad 
X^{-1} \A X \subseteq {\rm U}_n(K) \subseteq 
\left(
\begin{array}{cc}
{\rm M}_{r}(K)      & {\rm M}_{r \times s}(K) \\
0_{s \times r}      & {\rm M}_{s}(K)
\end{array}
\right).
$$
Invoking Lemma \ref{block triangular conjugation}, with $q=2$, 
we deduce that there exist matrices $X_{11} \in {\rm M}_r(K)$, $X_{12} \in {\rm M}_{r \times s}(K)$, and $X_{22} \in {\rm M}_s(K)$ such that
$$
X = \left(
\begin{array}{cc}
X_{11}         & X_{12} \\
0_{s \times r} & X_{22}
\end{array}
\right).
$$
Therefore, by (\ref{inverse of block triangular}),
$$
X^{-1} =\left(
\begin{array}{cc}
X^{-1}_{11}    & X'_{12} \\
0_{s \times r} & X^{-1}_{22}
\end{array}
\right),
$$
where $X'_{12} = -X_{11}^{-1} X_{12} X_{22}^{-1}$.
Let $
\left(
\begin{array}{cc}
aI_r & A_{12} \\
    & aI_s
\end{array}
\right)
$ be an arbitrary element of 
$\A$, with $a \in K$, $A_{12} \in {\rm M}_{r \times s}(K)$. Then
\begin{align*}
X^{-1}
\left(
\begin{array}{cc}
aI_r & A_{12} \\
    & aI_s
\end{array}
\right)
X&=
\left(
\begin{array}{cc}
X^{-1}_{11}    & X'_{12} \\
0_{s \times r} & X^{-1}_{22}
\end{array}
\right)
\left(
aI_n+
\left(
\begin{array}{cc}
0_r & A_{12} \\
    & 0_s
\end{array}
\right)
\right)
\left(
\begin{array}{cc}
X_{11}         & X_{12} \\
0_{s \times r} & X_{22}
\end{array}
\right)= \\
&=
aI_n+
\left(
\begin{array}{cc}
0_r & X_{11}^{-1} A_{12} X_{22} \\
    & 0_s
\end{array}
\right),
\end{align*}
and so $X^{-1} \A X \subseteq \A$. Since we can take any matrix of M$_{r \times s}(K)$ for $A_{12}$, it is not hard to show that the opposite inclusion also holds. This completes the proof.
\end{proof}

For an odd integer $n \geq 3$, the algebras $C_n^1(K)$ and $C_n^2(K)$ are distinct subalgebras of U$_n(K)$. Moreover, we have the following:

\begin{corollary}\label{isomnotconj}
Let $n$ be an odd integer greater than or equal to 3. 
  Then the commutative subalgebras 
$C_n^1(K)$ 
and $C_n^2(K)$
of M$_n(K)$ 
are not conjugated.
\end{corollary}
Note that Proposition \ref{conjugation of Cn} cannot be extended to the subalgebras presented in  $(\ref{commJacob6}) - (\ref{lastone})$: 
%
\begin{example}
Conjugation of the subalgebra 
$C_2^2(K) = \left\{ 
\left(
\begin{array}{cc}
a & 0 \\
0 & b 
\end{array}
\right)
 : a, b \in K \right\}
$ of M$_2(K)$ with the invertible matrix 
$X = 
\left(
\begin{array}{cc}
1 & 1 \\
0 & 1 
\end{array}
\right)
$
gives 
$$
X^{-1}C^2_2(K) X = 
\left(
\begin{array}{cc}
1 & -1 \\
0 & 1 
\end{array}
\right)
C^2_2(K)
\left(
\begin{array}{cc}
1 & 1 \\
0 & 1 
\end{array}
\right)
=
\left\{ 
\left(
\begin{array}{cc}
a & a-b \\
0 & b 
\end{array}
\right)
 : a, b \in K \right\},
$$
which is a subalgebra of U$_2(K)$ different from $C^2_2(K)$. 
\end{example}
Similar examples can also be found in case of the subalgebras $C_{3}^3(K)$, $C_{3}^4(K)$ and $C_{3}^5(K)$ of M$_3(K)$. Two of them can be already constructed from the previous analysis, while defining $C_{3}^3(K)$ by formula (\ref{commJacob6}) and $C_{3}^4(K)$ by formula (\ref{commJacob7}).

\smallskip

In the following sections we will use the algebras presented in $(\ref{commJacob0}) - (\ref{lastone})$  for arbitrary fields. If we need the assumption that $K$ is an algebraically closed field, then we will stress this assumption explicitly.

\section{Isomorphic {\rm D}$_q$ subalgebras of~{\rm M}$_n(K)$ 
with max-dim db's are of the same type} \label{isomorphism problem}

In this section, we will treat the isomorphism problem of D$_q$ subalgebras of M$_n(K)$ with max-dim db's. Before we state the main results we will give a characterization of the subalgebras we are dealing with.

Note that by Remark \ref{maxtuple}, for any maximal D$_q$ subalgebra of M$_n(K)$ there exists exactly one $q$-tuple $(n_1, n_2, \ldots, n_q)$, which indicates the type of conjugated D$_q$ subalgebra of~M$_n(K)$ with max-comm db's. So we can look at the class of all maximal D$_q$ subalgebras of M$_n(K)$ associated with the fixed tuple $(n_1, n_2, \ldots, n_q)$. In general, we are not able to say much about this class, even if we treat algebras up to conjugation. However, if we restrict it to these algebras with maximum possible dimension for the considered class, then it consists of subalgebras conjugated with a 
D$_q$~subalgebra of M$_n(K)$ of type~$(n_1, n_2, \ldots, n_q)$ with max-dim db's, which is more accessible. We want to stress that only for specific tuples, which will be described in Section \ref{sectionmaximal}, the restricted class consists of D$_q$~subalgebras of M$_n(K)$ with maximum dimension.

In Theorem \ref{theorem typical dk algebras} we prove that if two D$_q$ subalgebras 
of M$_n(K)$ with max-dim db's are isomorphic, then they
are of the same type and the algebras in their diagonal blocks 
are pairwise isomorphic. Next, we will show that this theorem cannot be inverted, which indicates what should be modified to completely solve the isomorphism problem. 
Finally, for an algebraically closed field $K$, we will prove that 
two D$_q$ subalgebras 
of M$_n(K)$ with max-dim db's are isomorphic
if and only if they are of the same type and  their diagonal blocks 
are pairwise conjugated. In contrast with the examples of isomorphic but not conjugated commutative subalgebras of M$_n(K)$ with maximum dimension, it implies that isomorphic D$_q$ subalgebras of M$_n(K)$ with maximum dimension are conjugated.

Recall that if we consider  a 
{\rm D}$_q$ subalgebra of~{\rm M}$_n(K)$ of type $(n_1, n_2, \ldots, n_q)$ with max-dim db's, then
we always use the notation related to (\ref{block triang}). In other words, for 
{\rm D}$_q$ subalgebras $\mathcal{A}$ and $\mathcal{B}$ of~{\rm M}$_n(K)$ of types $(n_1, n_2, \ldots, n_q)$ and $(\ell_1, \ell_2, \ldots, \ell_q)$, respectively, with max-dim db's, we will write
\begin{equation} \label{typical Dq}
\mathcal{A} = 
\left(
\begin{array}{ccccc}
\mathcal{A}_{11} & \mathcal{A}_{12}  & \ldots      & \mathcal{A}_{1q} \\ 
                 & \mathcal{A}_{22}  & \ldots      & \mathcal{A}_{2q}  \\
                 &                   & \ddots      & \vdots   \\
                 &                   &             &  \mathcal{A}_{qq}            
\end{array}
\right)
\quad \quad {\rm and} \quad\quad
\mathcal{B} = 
\left(
\begin{array}{ccccc}
\mathcal{B}_{11} & \mathcal{B}_{12}  & \ldots      & \mathcal{B}_{1 q} \\ 
                 & \mathcal{B}_{22}  & \ldots      & \mathcal{B}_{2 q}  \\
                 &                   & \ddots      & \vdots   \\
                 &                   &             &  \mathcal{B}_{qq}            
\end{array}
\right),
\end{equation}
where $\mathcal{A}_{ii}$ (respectively, $\mathcal{B}_{ii}$) is a commutative subalgebra of M$_{n_i}(K)$ (respectively, M$_{\ell_i}(K)$) with maximum dimension for every $i$, $i =1,2 \ldots,q,$ and $\mathcal{A}_{ij}= {\rm M}_{n_i \times n_j}(K)$ and $\mathcal{B}_{ij}={\rm M}_{\ell_i \times \ell_j}(K)$ for all~$i$ and~$j$ such that $1 \leq i < j \leq q$. Following the notation in Proposition \ref{proposition power of ideal generated by commutators}, we henceforth denote the ideal of $\mathcal{A}$ (respectively, $\mathcal{B}$) generated by all commutators $[x,y]$, with $x, y \in \mathcal{A}$ (respectively, $x, y \in \mathcal{B}$) by $\mathcal{C}_{\A}$ (respectively, $\mathcal{C}_{\B}$).

In order to build a proof of Theorem~\ref{theorem typical dk algebras}, we will need a technical lemma based 
on~Proposition~\ref{proposition power of ideal generated by commutators}. 
\begin{lemma} \label{image of diagonal block}
Let $\varphi \colon \mathcal{A} \to \mathcal{B}$ be an isomorphism of D$_q$ subalgebras $\A$ and $\B$ of M$_n(K)$ of types $(n_1, n_2, \ldots, n_q)$ and $(\ell_1, \ell_2, \ldots \ell_q)$, respectively, with max-dim db's. 
Then,  using the 
notation in~(\ref{typical Dq}),  
\begin{equation*}
\varphi
\left(
\left( 
\begin{array}{c|c|c}
 &        & \\ \hline
 & \mathcal{A}_{ii} & \\ \hline 
 &        & 
\end{array}
\right)
\right)
\subseteq
\left(
\begin{array}{c|cccc}
& \mathcal{B}_{1 i}  & \ldots & \mathcal{B}_{1q} \\
& \vdots             & \ddots & \vdots \\
& \mathcal{B}_{ii}   & \ldots & \mathcal{B}_{iq} \\ \hline
&          &         & 
\end{array}
\right)
\end{equation*}
for every $i, \ i=1,2,\ldots,q$.
\end{lemma}
\begin{proof}
It follows from 
$\varphi(\mathcal{C}_{\mathcal{A}}) = \mathcal{C}_{\mathcal{B}}$ that $\varphi(\mathcal{C}_{\mathcal{A}}^i) = \mathcal{C}_{\mathcal{B}}^i$ for every $i, \ i =1,2,\ldots,q,$ and so, by
Proposition~\ref{proposition power of ideal generated by commutators}, 
\[
\mathcal{C}^i_{\A} \cdot 
\left( 
\begin{array}{c|c|c}
 &        & \\ \hline
 & {\A}_{ii} & \\ \hline 
 &        & 
\end{array}
\right) = \{0_n\},
\]
implying that 
\begin{equation} \label{product is zero}
\mathcal{C}^i_{\B} \cdot \varphi
\left(
\left( 
\begin{array}{c|c|c}
 &        & \\ \hline
 & {\A}_{ii} & \\ \hline 
 &        & 
\end{array}
\right)
\right)
= \{0_n\}.
\end{equation}

Let $1 < i < q.$ We first show that the equality in (\ref{product is zero}) implies that
\begin{equation} \label{inclusion to lemma}
\varphi
\left(
\left( 
\begin{array}{c|c|c}
 &        & \\ \hline
 & {\A}_{ii} & \\ \hline 
 &        & 
\end{array}
\right)
\right)
\subseteq
\left(
\begin{array}{cccccc}
{\B}_{11} & \ldots  & {\B}_{1 i}  & {\B}_{1,i+1}  & \ldots  & {\B}_{1q} \\
       &  \ddots & \vdots   &  \vdots    & \ddots  & \vdots \\
       &         & {\B}_{ii}   & {\B}_{i,i+1}  & \ldots  & {\B}_{iq} \\ 
       &         &          &  0_{n_{i+1}} & \ldots  & 0_{n_{i+1}, n_q}  \\
       &         &          &            & \ddots  & \vdots  \\
       &         &          &            &         &  0_{n_q}
\end{array}
\right).
\end{equation}
To this end, let $Y =(y_{ij})$ be an arbitrary matrix in 
$\varphi
\left(
\left( 
\begin{array}{c|c|c}
 &        & \\ \hline
 & {\A}_{ii} & \\ \hline 
 &        & 
\end{array}
\right)
\right),
$ and let $j$ be any integer such that $\ell_1+\ell_2+\cdots+\ell_i <j \leq n$ (recall that $\ell_1+\ell_2+\cdots+\ell_q = n$). Then it follows again from~Proposition~\ref{proposition power of ideal generated by commutators} that $e_{1j} \in \mathcal{C}_{\B}^i$, and so by (\ref{product is zero}),
$$0_n= e_{1j} Y = \left(
\begin{array}{ccccc}
y_{j1} & y_{j2} & \ldots & y_{jn} \\
0      &  0     & \ldots & 0      \\
\vdots & \vdots & \ddots & \vdots \\
0      & 0      & \ldots & 0
\end{array}
\right).
$$
Hence, rows $\ell_1+\ell_2+\cdots+\ell_i+1, \ell_1+\ell_2+\cdots+\ell_i+2, \ldots, n$ of $Y$ are zero, which establishes (\ref{inclusion to lemma}). 

Similar arguments, starting from the equality 
$$
\left( 
\begin{array}{c|c|c}
 &        & \\ \hline
 & {\A}_{ii} & \\ \hline 
 &        & 
\end{array}
\right) \cdot
\mathcal{C}^{q-i+1}_{\A} = \{0_n\},
$$ 
can be used to show that columns $1$, $2$, \ldots, $\ell_1+\ell_2+\cdots+\ell_{i-1}$  columns of an arbitrary matrix in the image
$
\varphi
\left(
\left( 
\begin{array}{c|c|c}
 &        & \\ \hline
 & {\A}_{ii} & \\ \hline 
 &        & 
\end{array}
\right)
\right)
$
are zero. 

As far as the cases $i=1$ and $i=q$ are concerned, it is evident that the gist of the above arguments also shows that rows $\ell_1+1, \ell_1+2, \ldots, n$ of every matrix in 
$\varphi
\left(
\left( 
\begin{array}{c|c|c}
\mathcal{A}_{11} &        & \\ \hline
 &  & \\ \hline 
 &        & 
\end{array}
\right)
\right)$ are zero, and that columns 1,2,\ldots,$\ell_1+\ell_2+\cdots+\ell_{q-1}$ of
every matrix in 
$\varphi
\left(
\left( 
\begin{array}{c|c|c}
 &        & \\ \hline
 &  & \\ \hline 
 &        & \mathcal{A}_{qq}
\end{array}
\right)
\right)$
are zero, which completes the proof.
\end{proof}

\medskip

\begin{theorem} \label{theorem typical dk algebras}
Let $\A$ and $\B$ be {\rm D}$_q$ subalgebras of~{\rm M}$_n(K)$ of types $(n_1, n_2, \ldots, n_q)$ and $(\ell_1, \ell_2, \ldots, \ell_q)$, respectively, with max-dim db's. If $\A$ and $\B$ are isomorphic, then the $q$-tuples $(n_1, n_2, \ldots, n_q)$ and $(\ell_1, \ell_2, \ldots, \ell_q)$ are equal and, using the notation in (\ref{typical Dq}), the algebras $\A_{ii}$ and $\B_{ii}$ are isomorphic for every $i,  \ i =1, 2, \ldots, q$. 
\end{theorem}

\begin{proof}
We will use the notation in~(\ref{typical Dq}). Let $\varphi \colon \A \to \B$ be an isomorphism of the algebras $\A$ 
and~$\B$. 
As in Lemma \ref{image of diagonal block}, we have $\varphi(\mathcal{C}_{\A}) = \mathcal{C}_{\B}$,
which implies the induced isomorphism $\overline{\varphi} \colon \A/\mathcal{C}_{\A} \to \B/\mathcal{C}_{\B}$. 
By Proposition \ref{proposition power of ideal generated by commutators},
we identify the quotient algebras $\A/\mathcal{C}_{\A}$ and~$\B/\mathcal{C}_{\B}$ with the direct products $\A_{11} \times \A_{22} \times \cdots \times \A_{qq}$ and $\B_{11} \times \B_{22} \times \cdots \times \B_{qq}$ of the 
algebras~$\A_{ii}$ and $\B_{ii}, \ i=1,2,\ldots,q$, respectively.  
As $\mathcal{C}_{\B}$ comprises all matrices with zero entries in the diagonal blocks,  the inclusion
$$
\overline{\varphi}(0_{n_1} \times \ldots \times 0_{n_{i-1}} \times \A_{ii} \times 0_{n_{i+1}} \times \ldots \times 0_{n_q}) \subseteq
0_{\ell_1} \times \ldots \times 0_{\ell_{i-1}} \times \B_{ii} \times 0_{\ell_{i+1}} \times \ldots \times 0_{\ell_q}.
$$
follows from Lemma \ref{image of diagonal block}.
Similarly, Lemma \ref{image of diagonal block} applied to the inverse ${\varphi}^{-1}$ 
yields
$$
(\overline{\varphi})^{-1}(0_{\ell_1} \times \ldots \times 0_{\ell_{i-1}} \times \B_{ii} \times 0_{\ell_{i+1}} \times \ldots \times 0_{\ell_q}) \subseteq
0_{n_1} \times \ldots \times 0_{n_{i-1}} \times \A_{ii} \times 0_{n_{i+1}} \times \ldots \times 0_{n_q}.
$$
These two inclusions
imply the equality
$$
\overline{\varphi}(0_{n_1} \times \ldots \times 0_{n_{i-1}} \times \A_{ii} \times 0_{n_{i+1}} \times \ldots \times 0_{n_q}) =
0_{\ell_1} \times \ldots \times 0_{\ell_{i-1}} \times \B_{ii} \times 0_{\ell_{i+1}} \times \ldots \times 0_{\ell_q},
$$
and so the algebras $\A_{ii}$ and $\B_{ii}$ are isomorphic. 

Next, since $\A_{ii}$ and $\B_{ii}$ are commutative subalgebras of M$_{n_i}(K)$ and M$_{\ell_i}(K)$  (respectively) with maximum dimension, it follows from Schur's Theorem that
\begin{equation} \label{maximal dimension needed}
{\rm dim}_K \A_{ii} = 1+ \left \lfloor \frac{n^2_i}{4} \right \rfloor = 1+ \left \lfloor \frac{\ell^2_i}{4} \right \rfloor = {\rm dim}_K \B_{ii},
\end{equation}
which implies the equality $n_i = \ell_i$ for any $i = 1, 2, \ldots, q$. 
\end{proof}

Note that in the equality (\ref{maximal dimension needed}) the assumption that $\mathcal{A}_{ii}$ and $\mathcal{B}_{ii}$ are commutative subalgebras of matrices with maximum dimension is essential. 
By Theorem \ref{max up to conj} a conclusion retlated to that presented in Theorem \ref{theorem typical dk algebras} holds for D$_q$ subalgebras $\A$ and $\B$ of M$_n(K)$ of types $(n_1, n_2, \ldots, n_q)$ and  $(\ell_1, \ell_2, \ldots, \ell_q)$, respectively with max-comm db's, if we assume that  $\A$ and $\B$ are conjugated. In this regard, we pose the following two questions:


\begin{question}\label{Ques1}
Do there exist isomorphic D$_q$  subalgebras $\A$ and $\B$ of M$_n(K)$ of types $(n_1, n_2, \ldots, n_q)$ and  $(\ell_1, \ell_2, \ldots, \ell_q)$, respectively, with max-comm db's,  which are not conjugated?
\end{question}

\begin{question}\label{Ques2}
Do there exist isomorphic D$_q$  subalgebras $\A$ and $\B$ of M$_n(K)$ of types $(n_1, n_2, \ldots, n_q)$ and  $(\ell_1, \ell_2, \ldots, \ell_q)$, respectively, with max-comm db's,  such that  $n_i\neq \ell_i$ for at least one~$i$?
\end{question}

Note that the paragraph preceding the two questions above implies that a positive answer to Question \ref{Ques2} would also answer Question \ref{Ques1} in the positive.



The next part of this section will lead us to a full 
characterization of isomorphisms between {\rm D}$_q$ subalgebras of~{\rm M}$_n(K)$ with max-dim db's if the field $K$ is algebraically closed. 
Firstly, without this assumption on $K$, we will show that there exist non-isomorphic {\rm D}$_q$ subalgebras of~{\rm M}$_n(K)$ of the same type with max-dim db's, and so the converse of Theorem \ref{theorem typical dk algebras} does not hold.  

\begin{lemma} \label{nonisomorphic dq}
Let $\A$ and $\B$ be {\rm D}$_q$ subalgebras of~{\rm M}$_n(K)$ 
of type $(n_1, n _2, \ldots, n_q)$ with max-dim db's, such that, for some $j, \ n_j$ is odd, $n_j\ge3$ and the $j$-th diagonal blocks of algebras $\A$ and $\B$ are $C_{n_j}^{1}(K)$ and $C_{n_j}^{2}(K)$,  
respectively. Then $\A$ and $\B$ are not isomorphic.    
\end{lemma}

\begin{proof}
Let $n_{j1} = \left \lfloor \frac{n_j}{2} \right \rfloor $, $n_{j2} = \left \lfloor \frac{n_j}{2} \right \rfloor +1$. Then $n_{j1} \neq n_{j2}$ and by $(\ref{commJacob1}), (\ref{commJacob2}), (\ref{commJacob4})$ and
$(\ref{commJacob5})$, 
$$
C_{n_j}^1(K) = 
K I_{n_j} + 
\left(
\begin{array}{cc}
0_{n_{j1}}                 & {\rm M}_{n_{j1} \times n_{j2}}(K) \\
0_{n_{j2} \times n_{j1}} & 0_{n_{j2}}
\end{array}
\right)
\;\;$$
and
$$
C_{n_j}^2(K) = 
K I_{n_j} + 
\left(
\begin{array}{cc}
0_{n_{j2}}                 & {\rm M}_{n_{j2} \times n_{j1}}(K) \\
0_{n_{j1} \times n_{j2}} & 0_{n_{j1}}
\end{array}
\right).
$$

We first consider the case $q =2$ and $j=1$. 
Then the Jacobson radical $J(\A)$ of $\mathcal{A}$ satisfies
$$
\left( 
\begin{array}{cc|c}
{0_{n_{11}}}                 & {\rm M}_{n_{11} \times n_{12}}(K) & {\rm M}_{n_1 \times n_2}(K) \\ 
                             & {0_{n_{12}}}                      &                       \\ \hline
                             &                                   & 0_{n_2}
\end{array}
\right)
\subseteq J(\mathcal{A}) \subseteq
\left( 
\begin{array}{cc|c}
{0_{n_{11}}}                 & {\rm M}_{n_{11} \times n_{12}}(K) & {\rm M}_{n_1 \times n_2}(K) \\ 
                             & {0_{n_{12}}}                &                       \\ \hline
                             &                             & {\rm M}_{n_{2}}(K)
\end{array}
\right),
$$
and by Proposition \ref{proposition power of ideal generated by commutators},
$$
\mathcal{C}_{\mathcal{A}} =
\left( 
\begin{array}{c|c}
0_{n_{1}}     & {\rm M}_{n_1 \times n_2}(K) \\ \hline
              & 0_{n_{2}}
\end{array}
\right),
$$
implying that
$$
{\rm dim}_K (J(\A) \mathcal{C}_{\A}) = {\rm dim}_K
\left(
\begin{array}{c|c}
0_{n_{11} \times n_{1}}     & {\rm M}_{n_{11} \times n_2}(K) \\ \hline
                            & 0_{(n_{12}+n_{2}) \times n_{2}}
\end{array}
\right) = n_{11} n_2.
$$
Similarly,
${\rm dim}_K (J({\B}) \mathcal{C}_{\B}) = n_{12} n_2$. If the algebras $\A$ and $\B$ were isomorphic, then the images of the ideals $J({\A})$ and $\mathcal{C}_{\A}$ of $\A$ under any algebra isomorphism from $\A$ to $\B$ would be the ideals $J({\B})$ and~$\mathcal{C}_{\B}$ of $\B$, respectively, and since the respective dimensions would be equal, we would have that
$$
n_{11} n_2 = {\rm dim}_K (J(\A)\mathcal{C}_{\A}) = {\rm dim}_K (J(\B)\mathcal{C}_{\B}) = n_{12} n_2,
$$  
i.e., $n_{11} = n_{12}$; a contradiction. Therefore,
$\A$ and $\B$ are not isomorphic.

The case $q = 2$ and $j = 2$ is very similar to the above one. Instead of the dimensions of $J(\A)\mathcal{C}_{\A}$ and $J(\B)\mathcal{C}_{\B}$ we have to compare the dimensions of $\mathcal{C}_{\A}J(\A)$ and $\mathcal{C}_{\B}J(\B)$.

Now we assume that $q>2$. Suppose (for the contrary) that 
$\varphi \colon \B \to \A$ is an isomorphism, and let $\overline{\varphi} \colon \B/\mathcal{C}_{\B}^2 \to \A/\mathcal{C}_{\A}^2$ be the induced isomorphism. (A similar strategy was followed in the proof of~Theorem~\ref{theorem typical dk algebras}.) By $\overline{x}$ and $\overline{y}$ we denote the images of elements $x \in \A$ and $y \in \B$ in the quotient algebras $\A/\mathcal{C}^2_\A$ and $\B/\mathcal{C}^2_\B$, respectively. 
Consider the subspace 
$$
V_j = \overline{
\left( 
\begin{array}{c|c|c}
 &        & \\ \hline
 & J(\B_{jj}) & \\ \hline 
 &        & 
\end{array}
\right)
}
\overline{
\mathcal{C}_\B}
$$
of the quotient algebra $\B/\mathcal{C}^2_\B$. Since, by assumption, $\B_{jj} = C_{n_j}^{2}(K)$, we have that 
\begin{equation} \label{JBj}
J(\B_{jj}) = \left(
\begin{array}{cc}
0_{n_{j2}}                 & {\rm M}_{n_{j2} \times n_{j1}}(K) \\
0_{n_{j1} \times n_{j2}} & 0_{n_{j1}}
\end{array}
\right),
\end{equation}
and so by Proposition \ref{proposition power of ideal generated by commutators},
\begin{equation} \label{Vj}
V_j=
\overline{
\left( 
\begin{array}{c|c|c}
 &        & \\ \hline
 & J(\B_{jj}) & \\ \hline 
 &        & 
\end{array}
\right)
}
\overline{
\mathcal{C}_\B}
=
\overline{
\left( 
\begin{array}{c|c|c}
 &        & \\ \hline
 & J(\B_{jj}) & \\ \hline 
 &        & 
\end{array}
\right)
}
\cdot
\overline{
\left(
\begin{array}{ccccc}
0_{n_1} & \B_{12}  & 0_{n_1 \times n_3} & \ldots              & 0_{n_1 \times n_{q}}      \\
        & \ddots   & \ddots              & \ddots              & \vdots      \\ 
        &          &  \ddots             &    \ddots           & 0_{n_{q-2} \times n_{q}}      \\
        &          &                     &    \ddots        & \B_{q-1,q}   \\
        &          &                     &                     & 0_{n_q}
\end{array}
\right)
}. 
\end{equation}
We consider two possibilities, namely $j<q$ or $j=q$, as we did above with the case $q=2$. 

Firstly, let $j <q$. Then, by (\ref{Vj}),
\begin{equation} \label{dim Vj}
{\rm dim}_K V_j = 
{\rm dim}_K (J(\B_{jj}) \B_{j,j+1}) =   {\rm dim}_K \left( \left(
\begin{array}{cc}
0_{n_{j2}}                 & {\rm M}_{n_{j2} \times n_{j1}}(K) \\
0_{n_{j1} \times n_{j2}} & 0_{n_{j1}}
\end{array}
\right) \cdot {\rm M}_{n_j \times n_{j+1}}(K) 
\right)
= n_{j2}n_{j+1}.
\end{equation}
Under an isomorphism of algebras, nilpotent elements are mapped to nilpotent elements, and so, invoking Lemma \ref{image of diagonal block}, we have the inclusion 
\begin{equation} \label{inclusion}
\overline{\varphi}(V_j) = 
\overline{
\varphi
\left(
\left( 
\begin{array}{c|c|c}
 &        & \\ \hline
 & J(\B_{jj}) & \\ \hline 
 &        & 
\end{array}
\right)
\right)
}
\overline{
\mathcal{C}_\A}
\subseteq 
\overline{
\left(
\begin{array}{c|ccccc}
& \A_{1 j}     & \A_{1, j+1}   &\ldots  & \A_{1q}   \\
& \vdots       & \vdots       & \ddots & \vdots   \\
& \A_{j-1,j}   & \A_{j-1,j+1} &\ldots  & \A_{j-1,q}  \\
& J(\A_{jj})   & \A_{j,j+1}   & \ldots & \A_{jq} \\ \hline
&              &              &        &
\end{array}
\right)
} \overline{\mathcal{C}_A}, 
\end{equation}
where
\begin{equation} \label{JAj}
J(\A_{jj}) = J(C_{n_j}^1(K)) =  
\left(
\begin{array}{cc}
0_{n_{j1}}                 & {\rm M}_{n_{j1} \times n_{j2}}(K) \\
0_{n_{j2} \times n_{j1}} & 0_{n_{j2}}
\end{array}
\right).
\end{equation}
Since the diagonal blocks
of $\mathcal{C}_\A$ are zero and the product of such elements in the quotient algebra~$\A/\mathcal{C}_\A^2$ is zero, it follows from (\ref{JAj}) that 
$$
{\rm dim}_K
\left( \ 
\overline{
\left(
\begin{array}{c|ccccc}
& \A_{1 j}     & \A_{1, j+1}   &\ldots  & \A_{1q}   \\
& \vdots       & \vdots       & \ddots & \vdots   \\
& \A_{j-1,j}   & \A_{j-1,j+1} &\ldots  & \A_{j-1,q}  \\
& J(\A_{jj})   & \A_{j,j+1}   & \ldots & \A_{jq} \\ \hline
&              &              &        &
\end{array}
\right)
} \overline{\mathcal{C}_A}
\right) = {\rm dim}_K (J(\A_{jj}) \A_{j, j+1})
$$
\begin{equation} \label{other dim}
= {\rm dim}_K
\left(
\left(
\begin{array}{cc}
0_{n_{j1}}                 & {\rm M}_{n_{j1} \times n_{j2}}(K) \\
0_{n_{j2} \times n_{j1}} & 0_{n_{j2}}
\end{array}
\right) \cdot {\rm M}_{n_j \times n_{j+1}}(K) 
\right)
= n_{j1}n_{j+1}.
\end{equation}
We have thus found (see (\ref{dim Vj})) that ${\rm dim}_K V_j = n_{j2}n_{j+1}$ , and by (\ref{inclusion}) and (\ref{other dim}), ${\rm dim}_K \overline{\varphi}(V_j) \leq n_{j1}n_{j+1}$. 
However, these dimensions are equal, implying that $n_{j2} \leq n_{j1}$. This is a contradiction, since $n_{j2} = n_{j1}+1$. 

Lastly, let 
$j = q$. Now we need another argument, because in this case, by (\ref{Vj}), $V_j$ is the zero space (and so 
$\dim_K V_j = 0 = \dim_K \overline{\varphi}(V_j)$). Instead, we will compare the dimensions of the spaces~$\mathcal{C}_\A^{q-1} J(\A)$ and $\mathcal{C}_\B^{q-1} J(\B)$. Note that 
$$
\left(
\begin{array}{ccccc}
0_{n_1}          & \A_{12}     & \ldots         & \mathcal{A}_{1q} \\ 
                 & \ddots            & \ddots         & \vdots   \\
                 &                   &  0_{q-1}  & \mathcal{A}_{q-1,q}  \\
                 &                   &                &  J(\mathcal{A}_{qq})            
\end{array}
\right)
\subseteq
J(\A) \subseteq
\left(
\begin{array}{ccccc}
\mathcal{A}_{11} & \ldots            & \A_{1,q-1}    & \mathcal{A}_{1q} \\ 
                 &   \ddots          &  \vdots        & \vdots   \\
                 &                   &  \A_{q-1,q-1}  & \mathcal{A}_{q-1,q}  \\
                 &                   &                &  J(\mathcal{A}_{qq})            
\end{array}
\right)
$$
and
$$
\left(
\begin{array}{ccccc}
0_{n_1}          & \B_{12}     & \ldots         & \mathcal{B}_{1q} \\ 
                 & \ddots            & \ddots         & \vdots   \\
                 &                   &  0_{q-1}       & \mathcal{B}_{q-1,q}  \\
                 &                   &                &  J(\mathcal{B}_{qq})            
\end{array}
\right)
\subseteq
J(\B) \subseteq
\left(
\begin{array}{ccccc}
\mathcal{B}_{11} & \ldots            & \B_{1,q-1}    & \mathcal{B}_{1q} \\ 
                 &   \ddots          &  \vdots        & \vdots   \\
                 &                   &  \B_{q-1,q-1}  & \mathcal{B}_{q-1,q}  \\
                 &                   &                &  J(\mathcal{B}_{qq})            
\end{array}
\right),
$$
where $J(\A_{qq})$ and $J(\B_{qq})$ are as in (\ref{JAj}) and (\ref{JBj}), respectively, with $j = q$.
Then, by Proposition~\ref{proposition power of ideal generated by commutators}, it is not hard to show that ${\rm dim}_K (\mathcal{C}_\A^{q-1} J(\A)) = n_1 n_{q2}$ and ${\rm dim}_K (\mathcal{C}_\B^{q-1} J(\B)) = n_1 n_{q1}$. We conclude, as before, that
$n_{q2} = n_{q1}$. This contradiction completes the proof.
\end{proof}

\begin{theorem} \label{characterisation of typical Dq} 
Let $K$ be an algebraically closed field, and let $\mathcal{A}$ and $\mathcal{B}$ be {\rm D}$_q$ subalgebras of~{\rm M}$_n(K)$ of types $(n_1, n_2, \ldots, n_q)$ and $(\ell_1, \ell_2, \ldots, \ell_q)$, respectively, with max-dim db's. Then $\mathcal{A}$ and $\mathcal{B}$ are isomorphic if and only if the $q$-tuples $(n_1, n_2, \ldots, n_q)$ and $(\ell_1, \ell_2, \ldots, \ell_q)$ are equal and the diagonal blocks $\mathcal{A}_{jj}$ and~$\mathcal{B}_{jj}$ are pairwise conjugated for $j = 1, 2, \ldots, q$. Moreover, for every $j, \ j=  1, 2, \ldots, q,$ there is exactly  one  $k_j$ such that $\mathcal{A}_{jj}$ and $\mathcal{B}_{jj}$
are conjugated with $C^{k_j}_{n_j}(K)$. 
\end{theorem}
\begin{proof}
Firstly, assume that 
the $q$-tuples $(n_1, n_2, \ldots, n_q)$ and $(\ell_1, \ell_2, \ldots, \ell_q)$ are equal and that the diagonal blocks $\mathcal{A}_{jj}$ and~$\mathcal{B}_{jj}$ are pairwise conjugated for $j = 1, 2, \ldots, q$.
It follows from Theorem~\ref{max up to conj}  that $\A$ and $\B$ are conjugated, and hence isomorphic. 


Conversely, assume that $\A$ and $\B$ are isomorphic.
By Remark \ref{typical commutative}, each diagonal block $\A_{jj}$ (respectively, $\B_{jj}$) is conjugated by a matrix
$X_{jj} \in {\rm M}_{n_j}(K)$ (respectively, $Y_{jj} \in {\rm M}_{\ell_j}(K)$) with some subalgebra  $C_{n_j}^{s_j}(K)$ (respectively, $C_{\ell_j}^{t_j}(K)$). 
We stress that, for the algebra $C_{n_j}^{s_j}(K)$, the number $n_j$ is exactly the number appearing in the sequence $(n_1, n _2, \ldots, n_q)$ which is the type of $\A$; similarly for the algebra $C_{\ell_j}^{t_j}(K)$.
So, 
$X_{jj}^{-1} \A_{jj} X_{jj} = C_{n_j}^{s_j}(K)$ and $Y_{jj}^{-1} \B_{jj} Y_{jj} = C_{\ell_j}^{t_j}(K)$. Define the following block diagonal matrices:
$$X = 
\left(
\begin{array}{cccc}
X_{11} &        &        & \\
       & X_{22} &        & \\
       &        & \ddots & \\
       &        &        & X_{qq}
\end{array}
\right),
\quad
Y = 
\left(
\begin{array}{cccc}
Y_{11} &        &        & \\
       & Y_{22} &        & \\
       &        & \ddots & \\
       &        &        & Y_{qq}
\end{array}
\right).
$$
It follows from Proposition \ref{conjugation block triangular} that $\A ' = X^{-1}\A X$ and $\B ' = Y^{-1}\B Y$ are {\rm D}$_q$ subalgebras of~{\rm M}$_n(K)$ of types~$(n_1, n_2, \ldots, n_q)$ and $(\ell_1, \ell_2, \ldots, \ell_q)$, respectively, with max-dim db's.  
It is straightforward to check that the diagonal blocks of the algebras $\A'$ and $\B'$ are equal to $C_{n_j}^{s_j}(K)$ and $C_{\ell_j}^{t_j}(K)$, respectively.
Since $\A$ and $\B$ are isomorphic, so are $\A '$ and $\B '$. Therefore, by Theorem \ref{theorem typical dk algebras}, the $q$-tuples $(n_1, n_2, \ldots, n_q)$ and $(\ell_1, \ell_2, \ldots, \ell_q)$ are equal and the diagonal blocks $C_{n_j}^{s_j}(K)$ and $C_{\ell_j}^{t_j}(K)$ are isomorphic for $j=1,2,\ldots,q.$ By Remark~\ref{typical commutative} and the arguments preceding it, an isomorphism is possible only for equal blocks, or for pairs $C_{n_j}^{1}(K)$ and $C_{\ell_j}^2(K)$, or for pairs $C_{n_j}^{2}(K)$ and $C_{\ell_j}^1(K)$, with odd integer $n_j = \ell_j \geq 3$. However, in the case of at least one pair of distinct blocks, Lemma~\ref{nonisomorphic dq} can be applied, and so $\A '$ and $\B '$ are not isomorphic; a contradiction. Consequently, $C_{n_j}^{s_j}(K)=C_{\ell_j}^{t_j}(K)$ for $j = 1, 2, \ldots, q$, and so
$$
X_{jj}^{-1}\A_{jj}X_{jj} = Y_{jj}^{-1}\B_{jj} Y_{jj},
$$
from which we get $(X_{jj} Y_{jj}^{-1})^{-1} \A_{jj} X_{jj} Y_{jj}^{-1} = \B_{jj}$. Hence  the diagonal blocks $\A_{jj}$ and~$\B_{jj}$ are conjugated for $j = 1, 2, \ldots, q$, which completes the proof.
\end{proof}

Note that from Theorem \ref{characterisation of typical Dq} and Theorem \ref{max up to conj} it follows that D$_q$ subalgebras of M$_n(K)$ of some types with max-dim db's over an algebraically closed field K are isomorphic if and only if these subalgebras are conjugated.

\begin{remark} \label{sequence determining typical}
    For algebraically closed fields, we have another tool which can help to even better  describe  a {\rm D}$_q$ subalgebra~$\A$ of~{\rm M}$_n(K)$ of type $(n_1, n_2, \ldots, n_q)$ with max-dim db's,
    namely we can say that $\A$ is of type $\big((n_1, k_1), (n_2, k_2), \ldots, (n_q, k_q)\big)$ where 
    for every $j, \ j = 1, 2, \ldots, q$, the 
    number~$k_j$ appears as the superscript in $C_{n_j}^{k_j}(K)$.  It should be clear that  $k_j$ depends on $n_j$, as follows:
    \[
    k_j=\begin{cases}
     1,  \qquad\qquad\qquad\qquad \  {\rm if \ } n_j=1; \\
     1 {\rm \ or \ } 2, \quad\quad\qquad\qquad \ {\rm if \ } n_j=2; \\
      1 {\rm \ or \ } 2 {\rm \ or \ } 3 {\rm \ or \ } 4 {\rm \ or \ } 5, \, {\rm if \ } n_j=3; \\
     1, \qquad\qquad\qquad\qquad \ {\rm if \ } n_j\geq4 {\rm \ and \ } n_j {\rm \ is \ even}; \\
     1 {\rm \ or \ } 2, \quad\quad\qquad\qquad \ {\rm if \ } n_j\geq5 {\rm \ and \ } n_j {\rm \ is \ odd}. 
    \end{cases}
    \]
    Moreover, the $q$-tuples of ordered pairs $\big((n_1, k_1), (n_2, k_2), \ldots, (n_q, k_q)\big)$ determine all {\rm D}$_q$ subalgebras of~{\rm M}$_n(K)$ of type $(n_1, n_2, \ldots, n_q)$ with max-dim db's  
    up to conjugation (and isomorphism). 
     
\end{remark}

We have already shown that every D$_q$ subalgebra of U$_n(K)$ with maximum dimension is a max-dim D$_q$ subalgebra of M$_n(K)$ of  some type $(n_1, n_2, \ldots, n_q)$ with max-dim db's (see 
Corollary~\ref{upper triangular max}). With the help of Remark \ref{sequence determining typical} we can almost precisely say what the diagonal blocks of this subalgebra look like when the field $K$ is algebraically closed.  

\begin{corollary} \label{max dim alg closed}
Let $K$ be an algebraically closed field, and let $\A$ be a D$_q$ subalgebra of U$_n(K)$ with maximum dimension. 
Then $\A$ is a max-dim {\rm D}$_q$ subalgebra of~{\rm M}$_n(K)$ of some type $(n_1, n_2, \ldots, n_q)$ with max-dim db's, such that in form (\ref{block triang})
each diagonal block $\A_{ii}$, $i =1, 2, \ldots, q$, satisfies one of the following conditions:
\begin{enumerate}
    \item $\A_{ii}$ is equal to $C^1_{n_i}(K)$, with integer $n_i$ greater than or equal to 1,
    \item $\A_{ii}$ is equal to $C^2_{n_i}(K)$, with odd integer $n_i$ greater than or equal to 3,
    \item $\A_{ii}$ is conjugated with $C_{2}^2(K)$, $C_{3}^3(K)$, $C_3^4(K)$ or $C_3^5(K)$.
\end{enumerate}
\end{corollary}
\begin{proof}
By Corollary \ref{cormaintheorem} and Remark \ref{sequence determining typical}, there exists an invertible matrix $Y \in {\rm M}_n(K)$ such that $Y^{-1}\A Y$ is a max-dim {\rm D}$_q$ subalgebra of~{\rm M}$_n(K)$ of type $(n_1, n_2, \ldots, n_q)$ with the $C_{n_i}^{t_i}(K)$'s as max-dim db's. 
By hypothesis, $\A$ is contained in U$_n(K)$, and we have that $\A= Y(Y^{-1} \A Y) Y^{-1}$. 
By~Proposition~\ref{conjugation block triangular}, $ \A $ is a max-dim D$_q$ subalgebra of M$_n(K)$ of type $(n_1, n_2, \ldots, n_q)$ with max-dim db's conjugated with the $C_{n_i}^{t_i}(K)$'s. 
A conjugation of $C_{n_i}^{t_i}(K)$ is contained in U$_{n_i}(K)$, because $\A  \subseteq {\rm U}_n(K)$, and so the conditions 
in the statement of the corollary follow from Proposition \ref{conjugation of Cn}.
\end{proof}

\section{max-dim D$_q$ subalgebras of M$_n(K)$ with max-dim db's}\label{sectionmaximal}

 In this section, we will describe the $q$-tuples $(n_1, n_2, \ldots, n_q)$ such that $\A$ is a max-dim {\rm D}$_q$ subalgebra of~{\rm M}$_n(K)$ of type $(n_1, n_2, \ldots, n_q)$ with max-dim db's, and we will provide examples illustrating our study.
 

All the results in this section are based on the description in \cite[\,pages 251-253]{LM1}, including \cite[Lemma 12 and Lemma 13]{LM1}, where it was shown that if $n_1$, $n_2$, \ldots, $n_q$ are positive integers such that $n_1+n_2+ \cdots+n_q = n$, then a D$_q$ subalgebra of M$_n(K)$ of type $(n_1, n_2, \ldots, n_q)$ with max-dim db's has maximum dimension if and only if, for all $i$ and $j$,

\begin{equation}\label{at most one or two}
|n_i-n_j| =\begin{cases} 0 {\rm \ or \ }2, {\rm \ if \ both \ } n_i {\rm \ and \ } n_j {\rm \ are \ even};\\
0 {\rm \ or \ }1, {\rm \ otherwise}.\end{cases}
\end{equation}
We recall and  reformulate slightly the mentioned results in \cite{LM1} by starting in Proposition \ref{proposition typical D2} with a  description of max-dim D$_2$ subalgebras of M$_n(K)$ with max-dim db's, which follows directly 
from~(\ref{at most one or two}). 


\medskip


 
\begin{proposition} \label{proposition typical D2}
Let $n_1$ and $n_2$ be positive integers such that $n_1+n_2=n$, and consider 
a D$_2$~subalgebra of M$_n(K)$ of type $(n_1, n_2)$ with max-dim db's.
Then 
$\A$ is a D$_2$ subalgebra of M$_n(K)$ of maximum dimension if and only if one of the following possibilities occurs:
\begin{itemize}
    \item[a)]  $n$ is odd, and $(n_1,n_2) = (\lfloor \frac{n}{2} \rfloor,\lfloor \frac{n}{2} \rfloor +1)$ or $(\lfloor \frac{n}{2}\rfloor+1, \lfloor \frac{n}{2} \rfloor)$.
    
        \item[b)] $4 | n$ and $(n_1,n_2) = (\frac{n}{2},\frac{n}{2})$.
        
     \item[c)] $n=2$ and $(n_1,n_2) = (1,1)$.    
    
    \item[d)] $n\ge6$, $n \equiv 2 \ (\negthinspace\negthinspace\negthinspace\negthinspace\mod 4)$ and $(n_1,n_2) = (\frac{n}{2},\frac{n}{2})$ or $(\frac{n}{2}-1, \frac{n}{2}+1)$ or $(\frac{n}{2}+1,\frac{n}{2}-1)$.
\end{itemize}
\end{proposition}

Invoking Proposition \ref{proposition typical D2} and the examples of commutative subalgebras of M$_n(K)$  with maximum dimensions given in 
Section \ref{remarks on jacobson} (see $(\ref{commJacob0})-(\ref{lastone})$), we obtain the following examples of D$_2$ subalgebras of ${\rm M}_5(K)$ and ${\rm M}_6(K)$ with maximum dimension.

\begin{example} \label{example typical D2}
The max-dim D$_2$ subalgebras of ${\rm M}_5(K)$ with max-dim db's 
are of types $(2,3)$ or $(3,2)$, for example,
\begin{align*}
K (e_{11}+e_{22}) + K(e_{33}+e_{44}+e_{55})+
\left(
\begin{array}{cc|ccc}
0 & K & K & K & K \\
  & 0 & K & K & K\\ \hline
  &   & 0 & K & K\\
  &   &   & 0 & 0\\
  &   &   &   & 0
\end{array}
\right), \\
K (e_{11}+e_{22}+e_{33}) + K(e_{44}+e_{55})+
\left(
\begin{array}{ccc|cc}
0 & K & K & K & K \\
  & 0 & 0 & K & K\\
  &   & 0 & K & K\\ \hline
  &   &   & 0 & K\\
  &   &   &   & 0
\end{array}
\right).
\end{align*}

 In such a way we can construct 20 different max-dim D$_2$ subalgebras of M$_5(K)$ of types $(2,3)$ or~$(3,2)$ with max-dim db's equal to $C_2^i(K)$ or $C_3^j(K)$, with $i \in \{1,2 \}$ and $j \in \{1,2,3,4,5 \}$. 
 By~Theorem~\ref{theorem typical dk algebras} and Lemma \ref{nonisomorphic dq}, these 20 different max-dim D$_2$ subalgebras of M$_5(K)$ of types $(2,3)$ or~$(3,2)$ with max-dim db's are pairwise non-isomorphic.
 If $K$ is algebraically closed, then by Corollary \ref{cormaintheorem} and~Theorem~\ref{characterisation of typical Dq}, any 
 D$_2$ subalgebra of M$_5(K)$ with maximum dimension is conjugated with exactly one of them.  


Similarly, the max-dim D$_2$ subalgebras of M$_6(K)$ with max-dim db's 
are of types $(2,4)$, $(3,3)$ or~$(4,2)$. There are 29 such pairwise non-isomorphic subalgebras with diagonal blocks equal to~$C_2^{i}(K)$ or $C_3^j(K)$, with $i \in \{1,2 \}$ and $j \in \{1,2,3,4,5 \}$, or $C_4^1(K)$. 
%
\end{example}

The procedure of determining $q$-tuples $(n_1, n_2, \ldots, n_q)$, 
such that 
$n_i \leq n_j$ for all $i <j$ 
and such that a D$_q$ subalgebra of M$_n(K)$ of type $(n_1, n_2, \ldots, n_q)$ with max-dim db's is a D$_q$ subalgebra 
of~M$_n(K)$ with maximum dimension was discussed in Remark 15 of \cite{LM1}. It starts with determining numbers satisfying $|n_i - n_j| \leq 1$ (for $s=0$ and $t=0$ in the theorem below). Next, in some situations we can add 1 to some of the numbers in the $q$-tuple $(n_1, n_2, \ldots, n_q)$ and simultaneously subtract 1 from some of the others numbers in such a way that condition (\ref{at most one or two}) is satisfied. 

The introductory paragraphs in Section \ref{Section 1} show that a max-dim D$_q$ subalgebra of M$_n(K)$ of some type $(n_1, n_2, \ldots, n_q)$ with max-dim db's has dimension equal to the expression in (\ref{Domokoseq4}). It follows that a permutation which changes the (ordered) $q-$tuple $(n_1, n_2, \ldots, n_q)$ give rise to a max-dim D$_q$ subalgebra of M$_n(K)$ of another type with max-dim db's. 
Combining this observation with  \cite[Remark 15]{LM1} we get the following:

\begin{theorem} \label{typical and tuples}
Let $\A$ be a 
{\rm D}$_q$ subalgebra of~{\rm M}$_n(K)$ of type $(n_1, n_2, \ldots, n_q)$ with max-dim db's.
Write $n = q \left \lfloor \frac{n}{q} \right \rfloor +r$, where $r$ is the non-negative integer less than $q$ in the Division Algorithm. Then $\A$ is a D$_q$ subalgebra of M$_n(K)$ with maximum dimension if and only if there exists  a permutation $\sigma\in S_q$ such that $n_{\sigma(i)} \leq n_{\sigma(j)}$ for all $i < j$ and  one of two possibilities occurs:
\begin{itemize}
    \item[a)] $\left \lfloor \frac{n}{q} \right \rfloor$ is even and there exists a non-negative 
%
     integer $s$, with $s\leq\left \lfloor \frac{r}{2} \right \rfloor$, such that 
    \[n_{\sigma(i)}=\begin{cases}
    \left \lfloor \frac{n}{q} \right \rfloor, {\rm \ if \ } 1\le i \le q-r+s; \\ \\
    \left \lfloor \frac{n}{q} \right \rfloor+ 1, {\rm \ if \ } q-r+s< i \le q-s; \\ \\
    \left \lfloor \frac{n}{q} \right \rfloor+ 2, {\rm \ if \ }  q-s< i \le q.
    \end{cases}
    \]
    \item[b)] $\left \lfloor \frac{n}{q} \right \rfloor$ is odd and there exists a non-negative 
    integer $t$, with $t\leq\left \lfloor \frac{q-r}{2} \right \rfloor$, such that
    \[n_{\sigma(i)}=\begin{cases}
    \left \lfloor \frac{n}{q} \right \rfloor- 1, {\rm \ if \ } 1\le i\le t; \\ \\
    \left \lfloor \frac{n}{q} \right \rfloor, {\rm \ if \ } t< i\le q-r-t; \\ \\
    \left \lfloor \frac{n}{q} \right \rfloor+ 1, {\rm \ if \ }  q-r-t<i\le q.
    \end{cases}
    \]
%
\end{itemize}
\end{theorem}

We conclude with two examples illustration the two parts of Theorem \ref{typical and tuples}. 

\begin{example} \label{example typical Dq}
We will find all 5-tuples $(n_1, n_2, n_3, n_4, n_5)$ with $n_1 \leq n_2 \leq n_3 \leq n_4 \leq n_5$ resulting in a max-dim D$_5$ subalgebra of M$_{14}(K)$ of type $(n_1, n_2, n_3, n_4, n_5)$ with max-dim db's.

Using the notation in Theorem \ref{typical and tuples}, we have $n=14$, $q = 5$, $\left \lfloor \frac{n}{q} \right \rfloor =2 $ and $r = 4$. Then, for $s = 0, 1, 2$, the 5-tuples $(n_1, n_2, n_3, n_4, n_5)$ are respectively equal to
$$
(2, 3, 3, 3, 3), \quad (2, 2, 3, 3, 4) \quad {\rm and } \quad (2, 2, 2, 4, 4).
$$
If we do not assume that the $n_i$'s appear in increasing order, then we can permute them, which leads to~$4$, $\frac{5!}{2! \cdot 2!}-1 = 29$ and $\frac{5!}{3! \cdot 2!}-1 = 9$ more possibilities, respectively. 

Similarly, we will find all the 7-tuples $(n_1, n_2, \ldots, n_7)$ with $n_1 \leq n_2 \leq \ldots \leq n_7$ resulting in a max-dim D$_7$ subalgebra of M$_{22}(K)$ of type $(n_1, n_2, \ldots, n_7)$ with max-dim db's. 
Again, using the notation from in Theorem \ref{typical and tuples}, we have $n = 22$, $q = 7$, $\left \lfloor \frac{n}{q} \right \rfloor =3 $, $r = 1$. So, for $t =0, 1, 2, 3$, the 7-tuples $(n_1, n_2, n_3, n_4, n_5, n_6, n_7)$ are respectively equal to 
$$
(3, 3, 3, 3, 3, 3, 4), \; (2, 3, 3, 3, 3, 4, 4), \; (2, 2, 3, 3, 4, 4, 4) \; {\rm and} \; (2, 2, 2, 4, 4, 4, 4).
$$
In this case we have $6$, $\frac{7!}{4! \cdot 2!}-1 = 104$, $\frac{7!}{2! \cdot 2! \cdot 3!}-1 = 209$ and $\frac{7!}{3! \cdot 4!}-1=34$, respectively, more sequences if we don't assume that the $n_i$'s appear in increasing order.
\end{example}

\bigskip

\noindent\textbf{Acknowledgement}

\medskip

The authors would like to express their gratitude to Arkadiusz M\c{e}cel who provided Example~\ref{example block triangular} to them.


\medskip




\begin{thebibliography}{99}

\bibitem{Aus} M. Auslander, I. Reiten and S. O. 
Smal\o{}, \textit{Representation theory of Artin algebras. Cambridge Studies in Advanced Mathematics, 36}. Cambridge University Press, Cambridge, 1995.


\bibitem{BeiMi}
K. I. Beidar and A. V. Mikhalev, \textit{Generalized polynomial identities and rings which are sums of two subrings} (Russian),  Algebra i Logika  {\bf 34}(1) (1995), 3-11, 118;  English translation in  Algebra and Logic  {\bf 34}(1)  (1995), 1–5.


\bibitem{Domokos} M. Domokos, \textit{On the dimension of faithful modules over finite dimensional basic algebras}, 
Special issue on linear algebra methods in representation theory, Linear Algebra Appl.~{\bf 365} (2003), 155–-157.

\bibitem {Dren} V. Drensky, \textit{Free Algebras and PI-Algebras,} Springer-Verlag, 2000.

\bibitem {DrenForma} V. Drensky and E. Formanek, \textit{Polynomial Identity Rings,}
Birkh\"{a}user-Verlag, 2004.


\bibitem{Ja} N. Jacobson, \textit{Lie algebras. Interscience
Tracts in Pure and Applied Mathematics, No.~10} Interscience Publishers (a division of John Wiley \& Sons), New York - London, 1962.

\bibitem{Jac} N. Jacobson, \textit{Schur's theorems on commutative matrices}, Bull. Amer. Math. Soc. {\bf 50} (1944), 431--436.



\bibitem {Kemer} A. R. Kemer, \textit{Ideals of Identities of Associative Algebras,}
Translations of Math. Monographs, Vol.~87 (1991), AMS, Providence, Rhode Island.

\bibitem{Lam} T. Y. Lam, \textit{A First Course in Noncommutative Rings} (Second Edition), Graduate Texts in Mathematics {\bf 131}, Springer-Verlag, New York, 2001.  

\bibitem{Ma1} Yu.~N. Mal'tsev, \textit{A basis for the identities of the algebra of upper triangular matrices} (Russian),  Algebra i Logika {\bf 10} (1971), 393-400; English translation in Algebra and Logic {\bf 10} (1971), 242-247. 

\bibitem{Ma2} Yu.~N. Mal'tsev, \textit{Identities of matrix rings} (Russian), Sibirsk.~Mat.~Zh.~{\bf 22}(3) (1981), 213-214, 239.

\bibitem{MeSzvW}
J. Meyer, J. Szigeti and L. van Wyk, \textit{A Cayley-Hamilton trace identity for $2\times 2$ matrices over Lie-solvable rings}, Linear Algebra Appl. {\bf 436}  (2012), 2578-2582.


\bibitem{SIM} H. Radjavi and P. Rosenthal, \textit{Simultaneous Triangularization},  Universitext, Springer-Verlag, New York, 2000.

\bibitem{Sch} J. Schur, \textit{Zur Theorie der vertauschbaren Matrizen}, J.~Reine Angew.~Math.~{\bf 130} (1905), 66-76.

\bibitem{JSL} J. Szigeti, S. Szilágyi and L. van Wyk, \textit{A power Cayley-Hamilton identity for $n \times n$
matrices over a Lie nilpotent ring of index k}, Linear Algebra Appl.~{\bf 584} (2020), 153--163.

\bibitem{JJLM} J. Szigeti, J. van den Berg, L. van Wyk and M. Ziembowski, \textit{The maximum {\rm dim}ension of a Lie nilpotent subalgebra of {\rm M}$_n(F)$ of index m}, Trans.~Amer.~Math.~Soc.~\textbf{372}(7) (2019), 4553--4583.



\bibitem{LM1} L. van Wyk and M. Ziembowski, \textit{Lie solvability and the identity $[x_1 , y_1 ][x_2 , y_2 ] \cdots [x_q , y_q ] = 0$ in certain matrix algebras}, Linear Algebra Appl.~{\bf 533} (2017), 235--257.


\bibitem{LM2} L. van Wyk and M. Ziembowski, \textit{Lie solvability in matrix algebras}, Linear Multilinear Algebra {\bf 67}(4) (2019), 777--798.

\end{thebibliography}
\end{document}